%% file: preprint-submitted.tex
\documentclass[12pt]{article}
\usepackage{graphicx}
\usepackage{amsfonts,amssymb, hyperref}
\usepackage{amsmath,amsfonts,amsthm,amssymb}
\usepackage{float}

\usepackage{color}

\newtheorem{theorem}{Theorem}
\newtheorem{remark}{Remark}
\newtheorem{definition}{Definition}
\newtheorem{corollary}{Corollary}
\newtheorem{lemma}{Lemma}

\newtheorem{proposition}{Proposition}

\newtheorem{example}{Example}

\newcommand{\R}{\mathbb{R}}
\newcommand{\beq}{\begin{equation}}
\newcommand{\eeq}{\end{equation}}

\newcommand{\M}{\mathcal{M}}

\newcommand{\blue}{}
\newcommand{\bblue}{}

\title{\LARGE \bf 
Geometric and asymptotic properties associated with linear switched systems
\thanks{This research was partially supported by the iCODE institute, research project of the Idex Paris-Saclay.}
}

\author{Y. Chitour${}^\ddagger$, M. Gaye${}^\dagger$, P. Mason
\thanks{M. Gaye is with Centre de math\'ematiques appliqu\'ees, \'Ecole polytechnique, Palaiseau, France and the Laboratoire des signaux et syst\`emes, Supelec, Gif-sur-Yvette, France.
{\tt\small moussa.gaye@cmap.polytechnique.fr, moussa.gaye@lss.supelec.fr.}}%
\thanks{Y. Chitour and P. Mason are with the Laboratoire des signaux et syst\`emes, Universit\'e Paris-Sud, CNRS, Sup\'elec, Gif-sur-Yvette, France. 
{\tt\small yacine.chitour@lss.supelec.fr, paolo.mason@lss.supelec.fr.}}%
}

\begin{document}
\maketitle

\medskip

\begin{abstract} Consider continuous-time linear switched systems on $\mathbb{R}^n$ associated with compact convex sets of matrices. When the system is irreducible and the largest Lyapunov exponent is equal to zero, there always exists a Barabanov norm (i.e. a norm which is non increasing along trajectories of the linear switched system together with extremal trajectories starting at every point, that is trajectories of the linear switched system with constant norm). This paper deals with two sets of issues: $(a)$ properties of Barabanov norms such as uniqueness up to homogeneity and strict convexity; $(b)$ asymptotic behaviour of the extremal solutions of the linear switched system.
Regarding Issue $(a)$, we provide partial answers and propose four open problems motivated by appropriate examples. 
As for Issue $(b)$, we establish, when $n=3$, a Poincar\'e-Bendixson theorem under a regularity assumption on the set of matrices defining the system. Moreover, we revisit the noteworthy result of N.E. Barabanov \cite{Barabanov3} dealing with the linear switched system on $\mathbb{R}^3$ associated with a pair of Hurwitz matrices $\{A,A+bc^T\}$. 
We first point out a fatal gap in Barabanov's argument in connection with geometric features associated with a Barabanov norm. We then provide partial answers relative to the asymptotic behavior of this linear switched system.
\end{abstract}

\section{Introduction}
A continuous-time switched system is defined by a family $\mathcal{F}$ of continuous-time dynamical systems
and a rule determining at each time which dynamical system in $\mathcal{F}$ is 
responsible for the time evolution. In more precise terms, if $\mathcal{F}=\{f_u,\ u\in U\}$, where $U$ is a subset of $\mathbb{R}^m$, $m$ positive integer, and the $f_u$'s are sufficiently regular vector fields on a finite dimensional smooth  manifold $M$ one considers the family of dynamical systems $\dot x=f_u(x)$ with $x\in M$ and $u\in U$. The abovementioned rule is given by assigning a {\it switching signal}, i.e., a function 
$u(\cdot):\mathbb{R}_+\rightarrow U$, which is assumed to be at least measurable. One must finally define the class of admissible switching signals as the set of all rules one wants to consider. It can be for instance the set of all measurable 
$U$-valued functions or smaller subsets of it: piecewise continous functions, piecewise constant functions with positive {\it dwell-time}, i.e., the minimal duration between two consecutive  time discontinuities of $u(\cdot)$ (the so-called switching times) is positive, etc...
It must be stressed that an admissible switching signal is not known a priori and (essentially) represents a disturbance which is not possible to handle. This is why switched systems also fall into the framework of uncertain systems (cf \cite{Bubnicki} for instance). Switched systems are also called ``$p$-modal systems'' if the cardinality of $U$ is equal to some positive integer $p$, ``dynamical polysystems'', ``input systems''. They are said to be multilinear, or simply linear, when the state space is equal to $\mathbb{R}^n$, $n$ positive integer, and all the vector fields $f_u\in \mathcal{F}$ are linear. 
The term ``switched system'' was originally given to cases where the switching signal is piecewise constant or the set $U$ is finite: the dynamics of the switched system switches at switching times from one mode of operation $f_u$ to another one $f_v$, with $u,v\in U$. For a discussion of various issues related to switched systems, we refer to \cite{Liberzon,Margaliot,Shorten}.

A typical problem for switched systems goes as follows. Assume that, for every fixed $u\in U$, the dynamical system $\dot x = f_u(x)$ satisfies a given property $(P)$. Then one can investigate conditions under which Property $(P)$ still holds for every time-varying dynamical system $\dot x = f_{u(\cdot)}(x)$ associated with an arbitrary admissible switching signal $u(\cdot)$. The basic feature rendering that study non trivial comes from the remark that it is not sufficient for 
Property $(P)$ to hold uniformly with respect to (let say) the class of piecewise constant switching signals that it holds for each constant switching signal. For instance, one can find in \cite{Liberzon} an example in $\mathbb{R}^2$ of a linear switched system where admissible switching signals are piecewise constant and one switches between two linear dynamics defined by two matrices $A,B$. Then, one can choose $A,B$ both Hurwitz and switch in a such a way so as to go exponentially fast to infinity. Note also that one can do it with a positive dwell-time and therefore the issue of stability of switched systems is not in general related to switching arbitrarily fast.

In this paper, we focus on linear switched systems and Property $(P)$ is that of asymptotic stability (with respect to the origin). Let us first precise the setting of the paper. Here and after, a continuous-time linear switched system is described by a differential equation of the type
\begin{align}
\label{sys0}
\dot{x}(t)=A(t)x(t),\ t\geq 0,
\end{align}
where $x\in\mathbb{R}^n$, $n$ is a positive integer and $A(\cdot)$ is any measurable function taking values in a compact and convex subset $\mathcal{M}$ of $\mathbb{R}^{n\times n}$ (the set of $n\times n$ real matrices). 

There exist several approaches to address the stability issue for switched systems.  The most usual one consists in the search of a {\it common Lyapunov function} (CLF for short), i.e, a real-valued function which is a Lyapunov function for every dynamics $\dot x=f_u(x)$, $u\in U$. For the case of linear switched systems, that CLF can be chosen to be a homogeneous polynomial but, contrarily with respect  to the case where $\mathcal{M}$ is a singleton and one can choose the Lyapunov function to be a quadratic form, the degree of the polynomial CLF could be arbitrarily large (cf. \cite{mason1}). More refined tools rely on 
multiple and non-monotone Lyapunov functions, see for instance \cite{Branicki,Shorten}. Let us also mention linear switched systems technics based on the analysis of the Lie algebra generated by the matrices of  $\mathcal{M}$, cf. \cite{Agrachev}.

The approach we follow in this paper for investigating the stability issue of System~\eqref{sys0} is of more geometric nature and is based on the characterization of the worst possible behaviour of the system, i.e., one tries to determine a restricted  set of trajectories corresponding to some admissible switched signals (if any) so that the overall asymptotic behaviour of 
System~\eqref{sys0} is dictated by what happens for these trajectories.  It may happen that the latter reduces to a single one and we refer to it as the {\it worst-case trajectory}. This approach completely handles the stability issue for continuous-time two-dimensional linear switched systems when the cardinality of $\mathcal{M}$ is equal to two, cf. \cite{BBM,Boscain}. In higher dimensions, the situation is much more complicated. One must first  
consider the largest Lyapunov exponent of System~\eqref{sys0} given by
\begin{align}
\rho(\mathcal{M}) :=\sup\left( \limsup_{t\to+\infty}\frac{1}{t}\log\|x(t)\| \right),
\end{align}
where the supremum is taken over the set of solutions of~\eqref{sys0} associated with any  non-zero initial value and any switching law. 
Then, System~\eqref{sys0} is asymptotically stable (in the sense of Lyapunov) if and only if $\rho(\mathcal{M})<0$. In that case, one actually gets more, namely exponential asymptotic stability (i.e., there exist $\alpha>0$ and $\beta>0$ such that $\|x(t)\|\leq\alpha\exp(-\beta t)\|x(0)\|$ for every $t\geq0$ and for every solution $x(\cdot)$ of System~\eqref{sys0}) thanks to Fenichel lemma (see \cite{Colonius} for instance). On the other hand, \eqref{sys0} admits a solution which goes to infinity exponentially fast if and only if $\rho(\mathcal{M})>0$. Finally when $\rho(\mathcal{M})=0$ then either every solution of \eqref{sys0} starting from a bounded set remains uniformly bounded with one trajectory not converging to zero or System~\eqref{sys0} admits a trajectory going to infinity. \bblue{The notion of Joint Spectral Radius plays an analogous role for the description of the stability properties of discrete-time switched systems (cf.~\cite{Jungers} and references therein).} Note that when $\mathcal{M}$ reduces to a single matrix $A$, $\rho(A)$ is equal to the maximum real part of the eigenvalues of $A$. In the general case the explicit computation of $\rho(\M)$ is a widely open problem except for particular cases of linear switched systems in dimension less than or equal to $2$ (see e.g.~\cite{BBM,Boscain}).

Let us consider the subset $\mathcal{M}':=\{A-\rho(\mathcal{M})I_n:A\in\mathcal{M}\}$ of $\mathbb{R}^{n\times n}$, where $I_n$ denotes the identity matrix of $\mathbb{R}^{n\times n}$ and the continuous-time switched system corresponding to $\mathcal{M}'$. Notice that trajectories associated with $\M$ and trajectories associated with $\M'$ only differ at time $t$ by a scalar factor $e^{\rho(\M)t}$ and thus, in order to  understand the qualitative behaviour of trajectories of System~\eqref{sys0}, one  can always assume that $\rho(\mathcal{M})=0$, by eventually replacing $\M$ with $\M'$. Thus, this paper only deals with the case $\rho(\M)=0$.

The fundamental tool used to analyze trajectories of System~\eqref{sys0} is the concept of Barabanov norm (see \cite{Barabanov1,Wirth} and Definition~\ref{bnorm} below), which is well defined for irreducible sets of matrices. In that case recall that the value of a Barabanov norm decreases along trajectories of  \eqref{sys0} and, starting from every point $x\in\mathbb{R}^n$, there exists a trajectory of \eqref{sys0} along which a Barabanov norm is constant and such a trajectory is called an {\it extremal trajectory} of System~\eqref{sys0}.  Notice that the  concept of Barabanov norm can be extended to the case of discrete-time systems with spectral radius equal to $1$ (cf. \cite{Barabanov-fin,Wirth}).

By using the Pontryagin maximum principle, we provide a characterization of extremal trajectories first defined on a finite interval of time and then on an infinite one. Note that for the latter, this characterization was also derived in \cite{Barabanov1} by means of other technics. Moreover,
characterizing the points where a Barabanov norm $v$ is differentiable is a natural structural question.
In general, we can only infer from the fact that $v$ is a norm the conclusion that $v$ is differentiable almost everywhere on its level sets. We will provide a sufficient condition for differentiability of $v$ at a point $x\in\mathbb{R}^n$ in terms of the extremal trajectories reaching~$x$.

Another interesting issue is that of the uniqueness of Barabanov norms up to homogeneity (i.e., for every Barabanov norms $v_1(\cdot)$ and $v_2(\cdot)$ there exists $\mu>0$ such that $v_1(\cdot)=\mu v_2(\cdot)$). 
For discrete-time linear switched systems, the uniqueness of the Barabanov norms has been recently addressed  (cf. \cite{Morris1, Morris2} and references therein). Regarding continuous-time linear switched systems we provide a sufficient condition for  uniqueness, up to homogeneity, of the Barabanov norm  involving the $\omega$-limit set of extremal trajectories. We also propose an open problem which is motivated by an example of a two-dimensional  continuous-time linear switched system where one has an infinite number of Barabanov norms.

Recall that the Barabanov norm defined in \cite{Barabanov1} (see Equation~\eqref{norm} below) is obtained at every point $x\in \mathbb{R}^n$ by taking the supremum over all possible trajectories $\gamma$ of System~\eqref{sys0} of the limsup, as $t$ tends to infinity, of 
$\Vert \gamma(t)\Vert$, where $\Vert \cdot\Vert$ is a fixed vector norm on $\R^n$.
It is definitely a non trivial issue to determine whether this supremum is attained or not. If this is the case then the corresponding trajectory must be extremal. We provide an example in dimension two where the above supremum is not attained and propose an open problem in dimension three which asks whether the supremum is always reached if $\mathcal{M}$ is made of non-singular matrices. The above
mentioned issue lies at the heart of the gap in the proof of the main result of \cite{Barabanov3}. In that paper, one considers the linear switched system on $\mathbb{R}^3$ associated with a pair of matrices $\{A,A+bc^T\}$ where $A$ is Hurwitz, the vectors $b,c\ \in\R^3$ are such that the pairs $(A,b)$ and $(A^T,c)$ are both controllable and moreover the corresponding maximal Lyapunov exponent  is equal to zero (see for instance \cite{Margaliot} for a nice introduction to such linear switched systems and their importance in robust linear control theory). In the sequel,  we refer to such three dimensional switched systems as {\it Barabanov linear switched systems}.
Then, it is claimed in \cite{Barabanov3} that every extremal trajectory of a Barabanov linear switched system converges asymptotically to a unique periodic central symmetric extremal trajectory made of exactly four bang arcs.
In the course of his argument, N. E. Barabanov assumes that  the supremum in the definition of a semi-norm built similarly to a Barabanov norm is actually reached (cf. \cite[Lemma 9, page 1444]{Barabanov3}). Unfortunately there is no indication in the paper for such a fact to hold true and one must therefore conclude that the main result of \cite{Barabanov3} remains open.

We conclude the first part of the paper by mentioning another  feature concerning the geometry of Barabanov balls, namely that of their strict convexity. Indeed that issue is equivalent to the continuous differentiability of the dual Barabanov norm and it also has implications for the asymptotic behavior of the extremal trajectories. We prove that Barabanov balls are stricly convex in dimension two if $\mathcal{M}$ is made of non-singular matrices
and in higher dimension in case $\mathcal{M}$ is a $C^1$ domain of $\mathbb{R}^{n\times n}$. We also propose an open problem still  motivated by the example provided previously for the ``supremum-maximum''  issue and we ask whether Barabanov balls are strixtly convex under the assumption that $\mathcal{M}$ is made of non-singular matrices for $n>2$.
Note that several issues  described previously have been already discussed with less details in \cite{GCM}.

The second part of the paper is devoted to the analysis of the asymptotic behaviour of the extremal solutions of System~\eqref{sys0} in dimension three. Our first result consists in 
a Poincar\'e-Bendixson theorem saying that every extremal solution of System~\eqref{sys0} tends to a periodic solution of System~\eqref{sys0}. This result is obtained under a  regularity assumption on the set of matrices $\mathcal{M}$ (Condition~G) which is slightly weaker  than the analogous Condition~C  considered in~\cite{Barabanov2}. 
Note that \cite[Theorem 3]{Barabanov3} contains the statement of a result of Poincar\'e-Bendixson type similar to ours but the argument provided there (see page 1443) is extremely sketchy.

We then proceed by trying to provide a valid argument for Barabanov's result in \cite{Barabanov3}. We are not able to prove the complete statement 
but we can provide partial answers towards that direction. The first noteworthy result we get (see Corollary~\ref{cor1}) is that every periodic trajectory of Barabanov linear switched system is bang-bang with either two or four bang arcs. That fact has  an interesting numerical consequence namely that, in order to test the stability of the previous linear switched system,  it is enough to test products of at most four terms of the type $e^{tA}$ or $e^{t(A+bc^T)}$. Moreover, our main theorem (cf. Theorem~\ref{main-th} below) asserts the following. 
Consider a Barabanov linear switched system. Then, on the corresponding unit Barabanov sphere, there exists  a finite number of isolated periodic trajectories with at most four bangs and a finite number of $1$-parameter injective and continuous families of 
periodic trajectories with at most four bang arcs starting and finishing respectively at two distinct periodic trajectories with exactly two bang arcs. We actually suspect that such continuous families of periodic trajectories never occur, although at the present stage we are not able to prove it.
We also have a result describing the $\omega$-limit sets of any trajectory of a Barabanov linear switched system. We prove (see Proposition \ref{omega0}) that every trajectory of a Barabanov linear switched system either converges to a periodic trajectory (which can reduce to zero) or to the set union of a $1$-parameter injective and continuous family of periodic trajectories.

The structure of the paper goes as follows. In Section~\ref{s2}, we recall basic definitions of Barabanov norms and we  
provide a characterization of extremal trajectories (similar to that of \cite{Barabanov1}) by using the Pontryagin maximum principle. 
In Section~\ref{s3}, several issues are raised relatively to geometric properties of Barabanov norms and balls  such as uniqueness up to homogeneity of  the Barabanov norm and strict convexity of its unit ball. We also propose open questions.
We  state our Poincar\'e-Bendixson result in Section~\ref{s6} and we collect in Section~\ref{s-33} our investigations on Barabanov linear switched systems. 

\bigskip

{\bf Acknowledgements} The authors would like to thank E. Bierstone  and J. P. Gauthier for their help in the argument of Lemma~\ref{lem:accu}.
\subsection{Notations}  
If $n$ is a positive integer, we use $\mathbb{R}^{n\times n}$ to denote the set of $n$-dimensional square matrices with real coefficients, $A^T$ the matrix  transpose of an $n\times n$ matrix $A$, $I_n$ the $n$-dimensional identity matrix,
$x^T y$
the usual scalar product of $x, y\in\mathbb{R}^n$ and $\|\cdot\|$ the Euclidean norm on $\mathbb{R}^n$. 
We use $\R_+$, $\R^*$ and $\R_+^*$  respectively, to denote the set of non-negative real numbers, the set of non zero real numbers and the set of positive real numbers respectively. If $\M$ is a subset of $\mathbb{R}^{n\times n}$, we use $conv(\M)$ to denote the convex hull of $\M$. 
Given two points $x_0,x_1\in\mathbb{R}^n$ we will indicate as $(x_0,x_1)$ the open segment connecting $x_0$ with $x_1$. Similarly, we will use the bracket symbols ``$[$''  and ``$]$'' to denote left and right closed segments.
If $x\in\R$, we use $sgn(x)$ to denote $x/\vert x\vert$ if $x$ is non zero (i.e. the sign of $x$) and $[-1,1]$ if $x=0$.
If $A:\R_+\rightarrow \mathbb{R}^{n\times n}$ is measurable and locally bounded, the {fundamental matrix associated with} $A(\cdot)$ is the function $R(\cdot)$ solution of the Cauchy problem defined by $\dot{R}(t)=A(t)R(t)$ and $R(0)=I_n$.

\section{Barabanov norms and adjoint system}
\label{s2}

\subsection{Basic facts}
In this subsection, we collect basic definitions and results on 
Barabanov norms for linear switched system associated with a compact convex subset $\mathcal{M}\subset \R^{n\times n}$.
\begin{definition}
\label{def1}
We say that $\mathcal{M}$ (or System~\eqref{sys0}) is \textit{reducible} if there exists a proper subspace of $\mathbb{R}^n$ invariant with respect to every matrix $A\in\mathcal{M}$. Otherwise, $\mathcal{M}$ (or System~\eqref{sys0}) is said to be \textit{irreducible}. 
\end{definition}

We define the function 
$v(\cdot)$ on $\mathbb{R}^n$ as

\vspace{-14pt}

\begin{align}
\label{norm}
v(y):=\sup\left(\limsup_{t\rightarrow+\infty}\|x(t)\|\right),
\end{align}

\vspace{-5pt}

\noindent where the supremum is taken over all solutions $x(\cdot)$ of~\eqref{sys0} satisfying $x(0)=y$.
From~\cite{Barabanov1}, we have the following fundamental result.
\begin{theorem}[\cite{Barabanov1}]
\label{th1}
Assume that $\mathcal{M}$  is irreducible and $\rho(\mathcal{M})=0$. Then the function $v(\cdot)$ defined in~\eqref{norm} is a norm on $\mathbb{R}^n$ with the following properties:
\begin{enumerate}
\item[$1.$] for every solution $x(\cdot)$ of~\eqref{sys0} 
we have that $v(x(t))\leq v(x(0))$ for every $t\geq0$;
\item[$2.$] for every $y\in\mathbb{R}^n$, there exists a solution $x(\cdot)$ of~\eqref{sys0} starting at $y$ 
such that $v(x(t))=v(x(0))$ for every $t\geq0$.   
\end{enumerate}
\end{theorem}
\begin{definition}\label{def-ext} In the following, we list several definitions (see for instance \cite{Wirth}).
\begin{description}
\label{bnorm}
\item[-] A norm on $\mathbb{R}^n$ satisfying Conditions $1.$ 
and~$2.$ of Theorem~\ref{th1} is called 
a \textit{Barabanov norm}. Given such a norm $v(\cdot)$ we denote by $S:=\{x\in\mathbb{R}^n:v(x)=1\}$ the corresponding Barabanov unit sphere.
\item[-]  Given a Barabanov norm $v(\cdot)$ a solution $x(\cdot)$ of~\eqref{sys0} is said to be $v$\textit{-extremal} (or simply \textit{extremal} whenever the choice of the Barabanov norm is clear) if $v(x(t))=v(x(0))$ for every $t\geq0$.
\item[-] For a norm $w(\cdot)$ on $\mathbb{R}^n$ and  $x\in\mathbb{R}^n$, we use $\partial w(x)$ to denote the sub-differential of $w(\cdot)$ at $x$, that is the set of $l\in\mathbb{R}^n$ such that $l^T x=w(x)$ and $l^T y\leq w(y)$ for every $y\in\mathbb{R}^n$. {This is equivalent to saying that $l^T (y-x)\leq w(y)-w(x)$ for every $y\in\mathbb{R}^n$.}
\item[-] We define the $\omega$-limit set of a trajectory $x(\cdot)$ as follows:
$$
\omega(x(\cdot)):=\{\bar x\in\mathbb{R}^n:\exists (t_n)_{n\geq1}~\text{such that}~ t_n\to+\infty~\text{and}~x(t_n)\to\bar x~\text{as}~n\to+\infty\}.
$$
\end{description}
\end{definition}

\begin{remark}\label{classic}
It is easy to show that, given any trajectory of~\eqref{sys0}, its $\omega$-limit set $\omega(x(\cdot))$ is a non empty, compact and connected subset of $\mathbb{R}^3$. Indeed the argument is identical to the standard reasoning for trajectories of a regular and complete vector field in finite dimension. 
\end{remark}

In this paper we will be concerned with the study of properties of Barabanov norms and extremal trajectories. Thus we will always assume that $\M$ is irreducible and $\rho(\mathcal{M})=0$.

As stressed in the introduction, the study of extremal trajectories in the case $\rho(\mathcal{M})=0$ turns out to be useful for the analysis of the dynamics in the general case. Note that in the case in which $\rho(\mathcal{M})=0$ and $\M$ is reducible the system could even be unstable. A description of such instability phenomena has been addressed in~\cite{chitour}.

In the next proposition, we consider a norm which is dual to $v$ and gather its basic 
properties.
\begin{proposition}\label{prop-v*}(see \cite{Rock})
Consider the function $v^*$ defined on $\mathbb{R}^n$ as follows
\begin{equation}\label{v*}
v^*(l):=\max_{y\in S}{l^T y}.
\end{equation}
Then, the following properties hold true.
\begin{description}
\item[$(1)$] The function $v^*(\cdot)$ is a norm on $\mathbb{R}^n$.
\item[$(2)$] $S^0=\underset{x\in S}\cup\partial v(x)$, where  we use $S^0$ to denote  the unit sphere of $v^*$ (i.e. the polar of $S$). 
\item[$(3)$] For $x\in S$ and $l\in S^0$, we have that $l\in\partial v(x)$ if and only if $x\in\partial v^*(l)$.
\end{description}
\end{proposition}

\begin{remark}\label{s-conv}
 If $v$ (respectively $v^*$) is not differentiable at some $x\in S$ (respectively $l\in S^0$) then $S^0$ (respectively $S$) is not strictly convex since it contains a nontrivial segment included in $\partial v(x)$ (respectively in $\partial v^*(l)$).
\end{remark}

\subsection{Characterization of extremal trajectories}
In this section we characterize extremal trajectories by means of the Pontryagin maximum principle. Many of the subsequent results have been already established using other techniques  (see for instance~\cite{Barabanov1,Margaliot}).
\begin{definition}
Given a compact convex subset $\M$ of $\R^{n\times n}$, we define the adjoint system associated with~\eqref{sys0} as
\begin{align}
\label{ad}
\dot{l}(t)=-A^T(t)l(t),
\end{align}
where $A(\cdot)$ is any measurable function taking values in $\M$.
\end{definition}

\begin{lemma}
\label{milk}
For every solution $l(\cdot)$ of~\eqref{ad} and $t\geq0$, one has that $v^*(l(0))\leq v^*(l(t))$.
\end{lemma}
\begin{proof}
Let $l(\cdot)$ be a solution of System~\eqref{ad} associated with a switching law $A(\cdot)$. By definition of $v^*(l(0))$ there exists $z_0\in S$ such that $v^*(l(0))=l(0)^T z_0$. Consider the solution $z(\cdot)$ of System~\eqref{sys0} associated with $A(\cdot)$ and such that $z(0)=z_0$. Then $l(t)^T z(t)=l(0)^T z_0$ for every $t\geq 0$. Therefore, we get that $v^*(l(0))=l(t)^T z(t)=l(t)^T \big(\frac{z(t)}{v(z(t))}\big)v(z(t))$ for every $t\geq0$. Since $v(z(t))\leq v(z_0)=1$ and $l(t)^T \big(\frac{z(t)}{v(z(t))}\big)\leq v^*(l(t))$, we conclude that $v^*(l(0))\leq v^*(l(t))$ for every $t\geq0$. 

\end{proof}

\begin{theorem}
\label{mtheorem}
Let  $x(\cdot)$ be an extremal solution of $\eqref{sys0}$ such that $x(0)=x_0$ and associated with $\hat A(\cdot)$, $T\geq0$ and $\hat{l}\in\partial v(x(T))$. Then
there exists a non-zero solution $l(\cdot)$ of~\eqref{ad} associated with $\hat A(\cdot)$ such that the following holds true:
\begin{align*}
l(t)\in\partial v(x(t)),\quad &\forall t\in[0,T],\\
\max_{A\in\mathcal{M}}{l^T(t)Ax(t)}=l^T(t)\hat A(t)x(t)&=0, \quad a.e.\ t\in  [0,T],\\
l(T)=\hat{l}.&
\end{align*}
\end{theorem}
\begin{proof}
Let $x(\cdot)$ be as in the statement of the theorem and fix $T\geq 0$ and  $\hat{l}\in\partial v(x(T))$.
Define  $\varphi(z):=\hat{l}^T z$ for $z\in \mathbb{R}^n$. We consider the following optimal control problem in Mayer form (see for instance~\cite{Bressan})
\begin{equation}
\max {\varphi(z)},\label{pmax}
\end{equation}
among trajectories $z(\cdot)$ of \eqref{sys0} satisfying $z(0)=x_0$ and with free final time $\tau\geq 0$. 
Then, the pair $(x(\cdot),\hat A(\cdot))$ is an optimal solution of Problem~\eqref{pmax}.
Indeed, let $z(\cdot)$ be a solution of $\eqref{sys0}$ defined in $[0,\tau]$ such that $z(0)= x_0$. Since $\hat{l}\in \partial 
v(x(T))$, then $v(z(\tau))- v(x(T))\geq \hat{l}^T(z(\tau)-x(T))$. Since $v(z(\tau))\leq v(z(0))= v(x_0)=v(x(T))$, one gets 
$\hat{l}^T z(\tau)\leq\hat{l}^T x(T)$.

Consider the family of hamiltonians $h_{A}(z,l)=l^T Az$ where $(z,l)\in\mathbb{R}^{n}\times\mathbb{R}^{n}$ and $A\in\mathcal{M}$. 
Then 
the Pontryagin maximum principle ensures the existence of a nonzero Lipschitz map $l(\cdot):
\left[0,T\right]\rightarrow \mathbb{R}^n$ satisfying the following properties:
\begin{itemize}
\item[1)] $\dot{l}(t)=-\frac{\partial h_{\hat A(t)}}{\partial z}(x(t),l(t))= -\hat A^{T}(t)l(t),\quad a.e.\ t\in [0,T]$
\item[2)] $l^T(t)\hat A(t)x(t) = \max_{A\in\mathcal{M}}{l^T(t)Ax(t)} = 0,\quad a.e.\ t\in [0,T]$,
\item[3)] $l(T)=\hat{l}$.
\end{itemize}
To conclude the proof of Theorem $\ref{mtheorem}$, it is enough to show that $l(t)\in \partial v(x(t))$ for all $t\in\left[0,T\right]$. Indeed fix 
$t\in\left[0,T\right]$, $x\in \mathbb{R}^n$ and let $y(\cdot)$ be a solution of System $\eqref{sys0}$ such that $\dot{y}(\tau)= 
\hat A(\tau)y(\tau)$ with initial data $y(t)= x$. Then one has
\begin{align*}
v(x)-v(x(t))&\geq v(y(T))- v(x(t))= v(y(T))- v(x(T))
\\&\geq l^T(T)(y(T)- x(T))= l^T(t)(y(t)-x(t)),
\end{align*}
since $v(y(T))\leq v(y(t))= 
v(x)$, $l(T)\in\partial v(x(T))$ and the function $l^T(t)(y(t)-x(t))$ is constant on $[0,T]$. Hence,
$v(x)-v(x(t))\geq l^T(t)(x-x(t))$. Since $t\in\left[0,T\right]$ and $x\in\mathbb{R}^n$ are arbitrary, one deduces that $l(t)\in \partial v(x(t))$ for all $t\in\left[0,T\right]$. 

\end{proof}

A necessary condition for being an extremal for all non negative times is given in the  following result, which has already been established in \cite[Theorem 4]{Barabanov1}.
We provide here an argument based on Theorem~\ref{mtheorem}.
\begin{theorem}
\label{th2}
For every extremal solution $x(\cdot)$ of $\eqref{sys0}$ associated with a switching law $A(\cdot)$,
there exists a nonzero solution $l(\cdot)$ of~\eqref{ad} associated with $A(\cdot)$ such that 
\begin{align}
\label{c3}
l(t)&\in\partial v(x(t)),\quad \forall t\geq 0,\\
\label{c4}
\max_{A\in\mathcal{M}}{l^T(t)Ax(t)}&=l^T(t)A(t)x(t)=0,\quad a.e.\ t\geq 0.
\end{align}
\end{theorem}
\begin{proof}
Let $x(\cdot)$ be an extremal solution of System $\eqref{sys0}$ with some switching law $A(\cdot)$. For every $k\in\mathbb{N}$, let $\hat{l}_{k}\in\partial v(x(k))$ and $l_k(\cdot)$ solution of $\eqref{ad}$ associated with $A(\cdot)$  defined in Theorem $\ref{mtheorem}$. Since the sequence $l_k(0)\in\partial v(x(0))$ for all $k$ and $\partial v(x(0))$ is compact then up to subsequence $l_k(0)\rightarrow l_*\in\partial v(x(0))$. We then consider the solution $l(\cdot)$ of $\eqref{ad}$ associated with $A(\cdot)$ and such that $l(0)=l_*$. The sequence of solutions $l_k(\cdot)$ converges uniformly to $l(\cdot)$ on every compact time interval of $\mathbb{R}_+$. Hence by virtue of the compactness of $\partial v(x(t))$  and the fact that $l_k(t)\in\partial v(x(t))$ for every $t$ and $k$ sufficiently large, we get $l(t)\in\partial v(x(t))$.\\
For what concerns Eq.$\eqref{c4}$ notice that for $k$ sufficiently large, for every $t$ and for every $A\in\mathcal{M}$ we have ${l_k}^T(t)Ax(t)\leq {l_k}^T(t)A(t)x(t)= 0$. Then passing to the limit as $k\rightarrow+\infty$ we get that $l^T(t)Ax(t)\leq l^T(t)A(t)x(t)= 0$, which concludes the proof of the theorem. 

\end{proof}

\begin{remark}\label{adj11}
Note that along an extremal trajectory $x(\cdot)$, the curve $l(\cdot)$ defined previously takes values in $S^0$ thanks to Item $(2)$
of Proposition~\ref{prop-v*}.
\end{remark}

We now introduce an assumption on the set of matrices that will be useful in the sequel. 
\begin{definition}[\textbf{Condition G}]\label{condG}
For every non-zero $x_0\in\mathbb{R}^n$ and $l_0\in\mathbb{R}^n$, the solution $\big(x(\cdot),l(\cdot)\big)$ of $\eqref{sys0}$-$\eqref{ad}$ such that $\dot{x}(t)=A(t)x(t)$, $\dot{l}(t)=-A^T(t)l(t)$, $\big(x(0),l(0)\big)=(x_0,l_0)$ for some $A(\cdot)$ and satisfying $\max_{A\in\mathcal{M}}{l^T(t)Ax(t)}=l^T(t)A(t)x(t)$ for every $t\geq0$ is unique.
\end{definition}
\begin{remark} Barabanov introduced in~\cite{Barabanov2} a similar but slightly stronger condition referred as \textit{Condition~C}. Indeed, this condition requires a uniqueness property not only for solutions of $\eqref{sys0}$-$\eqref{ad}$ (as in Condition~G above) but also when $\M$ is replaced by $\mathcal{M}^T:=\{A^T: A\in\mathcal{M}\}$.
\end{remark}
In the remainder of the section, we establish  results dealing with regularity properties of Barabanov norms and the uniqueness of an extremal trajectory associated with a given initial point of $S$. Note that most of these results are either stated or implicitly used in 
\cite{Barabanov2} without any further details.

\begin{proposition}
\label{p3}
Let $x(\cdot)$ be an extremal solution of~\eqref{sys0} starting at some point of differentiability $x_0$ of $v(\cdot)$. Then the following results hold true:
\begin{enumerate}
\item The norm $v(\cdot)$ is differentiable at $x(t)$ for every $t\geq0$.
\item The solution $l(\cdot)$ of $\eqref{ad}$ satisfying Conditions $\eqref{c3}$-$\eqref{c4}$ of Theorem $\ref{th2}$ is unique and $l(t)=\nabla v(x(t))$ for $t\geq0$.
\item If Condition $G$ holds true, then $x(\cdot)$ is the unique extremal solution of $\eqref{sys0}$ starting at $x_0$.
\end{enumerate}
\end{proposition}
\begin{proof}
We first prove Item $1$. Let $x(\cdot)$ be an extremal solution of System $\eqref{sys0}$ associated with a switching law $A(\cdot)$ such that $x(0)=x_0$. Assume that $v$ is differentiable at $x_0$. Let $T>0$, $\hat{l}_1$ and $\hat{l}_2\in\partial v(x(T))$. By virtue of Theorem $\ref{mtheorem}$, there exist non-zero Lipschitz continuous solutions $l_1(\cdot)$ and $l_2(\cdot)$ of System $\eqref{ad}$ such that 
\begin{align*}
\dot{l_k}(t)=-A^{T}(t)l_k(t),\quad &a.e.\ t\in[0,T],\\
l_k(t)\in\partial v(x(t)),\quad &\forall t\in[0,T],\\
l_k(T)=\hat{l}_k,&
\end{align*}
for $k=1,2$. Since the norm $v$ is differentiable at $x_0$, then $l_k(0)=\nabla v(x_0)$ for $k=1,2$. Therefore by uniqueness of the Cauchy problem $\dot{l}(t)=-A^T(t)l(t)$ with $l(0)=\nabla v(x_0)$ we have $l_1(T)=l_2(T)$, i.e., $\hat{l}_1=\hat{l}_2$. Hence $v$ is differentiable at $x(T)$.

We now prove Item $2$. Notice that the solutions $l(\cdot)$ of  $\eqref{ad}$ satisfying Theorem $\ref{th2}$ verify  $l(t)\in\partial v(x(t))$ for every $t\geq 0$. According to Item $1$, we have $\partial v(x(t))=\{
\nabla v(x(t))
\}$ for all $t\geq0$. Hence, $l(\cdot)$ is unique and $l(t)=\nabla v(x(t))$ for all $t\geq0$.

To conclude the proof of Proposition $\ref{p3}$, we assume that Condition $G$ holds true. Let $y(\cdot)$ be an extremal solution starting at $x_0$. By Theorem $\ref{th2}$ and by Condition $G$ applied to $(x_0,\nabla v(x_0))$, we get $x(\cdot)=y(\cdot)$.

\end{proof}

Regularity of Barabanov norms is an interesting and natural issue to address.
For $x_0\in S$ we define the subset $\mathcal{V}_{x_0}$ of $\mathbb{R}^n$ by
\begin{equation}
\mathcal{V}_{x_0}:=\left\{l \in \mathbb{R}^n : 
\begin{array}{l}
\exists t_0>0,\exists x(\cdot)~\text{extremal with } x(t_0) = x_0,\\ (t_j)_{j\geq 1}\text{s.t.}~\displaystyle{\lim_{j\rightarrow+\infty} t_j=t_0^-}\text{and }l=\lim_{j\rightarrow+\infty}\frac{x(t_j)-x_0}{t_j-t_0}
\end{array}
\right\}.
\end{equation}
\begin{remark} 
The set $\mathcal{V}_{x_0}$ can be empty since Theorem~\ref{th1} does not guarantee the existence of extremal trajectories reaching $x_0$. 
\end{remark}
The next proposition provides a sufficient condition  for the differentiability of $v(\cdot)$  at a point $x_0$.
\begin{proposition}
\label{mar}
Let $x_0\in S$. Assume that $\mathcal{V}_{x_0}$ is not empty and contains $(n-1)$ linearly independent elements. Then $v(\cdot)$ is differentiable at $x_0$.
\end{proposition}

\begin{proof}
{
We first show that $\mathcal{V}_{x_0}\subset\mathcal{M}x_0:=\{Ax_0:A\in\mathcal{M}\}$. Indeed, assume by contradiction that there exists $y\in\mathcal{V}_{x_0}$ and $y\notin\mathcal{M}x_0$. Thus, by virtue of Hahn-Banach theorem, there exist $w\in\mathbb{R}^n$ and $\eta\in\mathbb{R}$ such that  $ w^T y<\eta<w^T Ax_0$ for all $A\in\mathcal{M}$. Thanks to the definition of $\mathcal{V}_{x_0}$ there exist $t_0>0$ and an extremal solution
$x(\cdot)$ of $\eqref{sys}$ for some switching law $A(\cdot)$ such that $x(t_0)=x_0$ and a sequence $(t_j)_{j\geq1}$ such that $t_j\rightarrow t_0^-$ as $j$ goes to infinity satisfying 
$y=\lim_{j\rightarrow+\infty}\frac{x(t_j)-x_0}{t_j-t_0}$.
We also have 
$$
\frac{x(t_j)-x_0}{t_j-t_0}=\frac{1}{t_j-t_0}\int_{t_0}^{t_j}{A(\tau)x(\tau)d\tau}=\frac{1}{t_j-t_0}\int_{t_0}^{t_j}{A(\tau)(x(\tau)-x_0)d\tau}+\frac{1}{t_j-t_0}\int_{t_0}^{t_j}{A(\tau)x_0d\tau},
$$ 
and 
$$
\lim_{j\rightarrow+\infty}\frac{1}{t_j-t_0}\int_{t_0}^{t_j}{A(\tau)(x(\tau)-x_0)d\tau}=0,
$$
so that $y=\lim_{j\rightarrow+\infty}\frac{1}{t_j-t_0}\int_{t_0}^{t_j}{A(\tau)x_0d\tau}$. Notice that for each $\tau\in(t_j,t_0)$, $w^T A(\tau)x_0>\eta$. Hence, by integrating and passing to the limit as $j$ goes to infinity, we get $w^T y\geq\eta$, leading to a contradiction. Thus $\mathcal{V}_{x_0}\subset\mathcal{M}x_0$.

Let us now check that if $l_0\in\partial v(x_0)$ and $y\in\mathcal{V}_{x_0}$ then $l_0^T y=0$. By Theorem~\ref{mtheorem} one has that   $l_0^T Ax_0\leq 0$ for each $A\in\mathcal{M}$, and in particular $l_0^T y\leq 0$. For the opposite inequality it is enough to observe that $l_0^T(x(t)-x_0)\leq v(x(t))-v(x_0)=0$ for any extremal trajectory $x(\cdot)$ and thus $l_0^T\big(\frac{x(t)-x_0}{t-t_0}\big)\geq 0$ for $t<t_0$. The desired inequality is obtained passing to the limit along the sequence $(t_j)_{j\geq 1}$.

Finally, under the assumptions of Proposition~\ref{mar}, there is a unique vector $l_0$ satisfying $l_0^T y=0$ for every $y\in\mathcal{V}_{x_0}$ and such that $l_0^Tx_0=v(x_0)=1$.
Hence $\partial v(x_0)=\{l_0\}$ and the differentiability of $v(\cdot)$ at $x_0$ is proved.
}

\end{proof}
Notice that in the particular case $n=2$ the previous result states that $v(\cdot)$ is differentiable at any point reached by an extremal trajectory. For $n=3$ differentiability at $x_0$ is instead guaranteed as soon as two extremal trajectories reach $x_0$ from two different directions. If the two incoming extremal trajectories reach $x_0$ 
with the same direction, we can still prove differentiability of $v$ at $x_0$ under the additional assumption that $\mathcal {M}$ is made of non-singular matrices, as shown below.

\begin{definition}
We say that two extremal solutions $x_1(\cdot)$ and $x_2(\cdot)$ of System~\eqref{sys0} intersect each other if there exist $t_1,t_2>0$ and $\epsilon>0$ such that $x_1(t_1)=x_2(t_2)$ and $x_1(s_1)\neq x_2(s_2)$ for every $s_1\in[t_1-\epsilon,t_1)$ and $s_2\in[t_2-\epsilon,t_2)$.
\end{definition}

\begin{proposition}
\label{p4} 
Assume that $n=3$ and every matrix of $\mathcal{M}$ is non-singular. If two extremal solutions $x_1(\cdot)$ and $x_2(\cdot)$ of~\eqref{sys0} intersect each other at some $z\in\mathbb{R}^3$, then $v(\cdot)$ is differentiable at $z$. If Condition $G$ holds, one has forwards uniqueness for extremal trajectories starting from $z$. 
\end{proposition}
\begin{proof} Let $x_1(\cdot)$ respectively $x_2(\cdot)$ be two extremal solutions of $\eqref{sys0}$ associated with $A_1(\cdot)$ respectively $A_2(\cdot)$, $t_1>0$ and $t_2>0$ such that $z$:=$x_1(t_1)$=$x_2(t_2)$. Since $\rho(\mathcal{M})=0$, we claim that $\mathcal{M}z\cap\mathbb{R}_+z=\emptyset$. If it were not the case, then there would exist $\lambda\geq0$ and $A\in\mathcal{M}$ such that $Az=\lambda z$. Notice that $\lambda>0$ since $A$ is non-singular. Consider the solution of System~\eqref{sys0} $e^{At}z=e^{t\lambda}z$ for all $t\geq0$. Then $\rho(\mathcal{M})\geq\lambda>0$ contradicting $\rho(\mathcal{M})=0$. This proves the claim. Notice that the set $\mathcal{M}z$ is compact and convex while $\mathbb{R}_{+}z$ is closed and convex, then there exist $w\in\mathbb{R}^n$, $\beta\in\mathbb{R}$ such that $\lambda w^T z<\beta<w^T Az$ for all  
$A\in\mathcal{M}$ and  $\lambda\in\mathbb{R}_+$. In particular $\beta>0$, $w^T z\leq 0$ and $w^T Az>\beta>0$ for all $A\in\mathcal{M}$.  

Since $\mathcal{M}$ is compact, one deduces by a standard continuity argument that, for $i=1,2$, there exists $0\leq\eta_i\leq t_i$ such that the map $t\in(t_i-\eta_i,t_i)\to w^T x_i(t)$ is strictly increasing.
Let $\xi>0$ be sufficiently small such that $w^T z-\xi> w^T x_i(t_i-\eta_i)$ for $i=1,2$. Then, there exists a unique $(\tau_1,\tau_2)\in(t_1-\eta_1,t_1)\times (t_2-\eta_2,t_2)$ such that $w^T z-\xi= w^T x_i(\tau_i)$ for $i=1,2$. Denote by $I:=\{x\in S: w^T x=w^T z-\xi\}$, $\mathcal{C}$ the shortest arc in $I$ joining $x_1(\tau_1)$ and $x_2(\tau_2)$, $\Gamma$ the closed curve constituted of $\mathcal{C}$, $\mathcal{P}_1$ and $\mathcal{P}_2$ where $\mathcal{P}_i:=\{x_i(t):t\in[\tau_i,t_i]\}$ for $i=1,2$ and finally we denote  by $\mathcal{T}$ the union of $\Gamma$ and its interior.

We next claim that if $\bar{x}$ is a point of differentiability of $v(\cdot)$ in $\mathcal{T}\setminus\Gamma$ and $\bar{\gamma}(\cdot)$ is an extremal solution of System $\eqref{sys0}$ such that $\bar{\gamma}(0)=\bar{x}$, then $\bar{\gamma}$ intersects $\mathcal{P}_1\cup\mathcal{P}_2$ (see Figure~\ref{preuve1}). To see that, first notice that there exists some $\bar{t}>0$ such that $\bar{\gamma}(\bar{t})\in\Gamma$. (Indeed, otherwise $\bar{\gamma}(t)\in \mathcal{T}\setminus\Gamma$ for all $t\geq 0$, which is not possible since the map $t\rightarrow w^T\bar{\gamma}(t)$ is uniformly strictly increasing in $\mathcal{T}\setminus\Gamma$.) Finally $\bar{\gamma}(\bar{t})\notin\mathcal{C}$ since $w^T\bar{\gamma}(\bar{t})> w^T\bar{\gamma}(0)\geq w^T z-\xi$. This proves that $\bar{\gamma}(\bar{t})\in \mathcal{P}_1\cup\mathcal{P}_2$.

We conclude the proof of Proposition~\ref{p4} by assuming that $\bar{\gamma}(\bar{t})\in\mathcal{P}_1$. Then $\bar{\gamma}(\bar{t})=x_1(\bar{t}_1)$ for some $\bar{t}_1\in(\tau_1,t_1)$. Since $v(\cdot)$ is differentiable at $\bar{x}$ then by Proposition $\ref{p3}$ $v(\cdot)$ is differentiable at $\bar{\gamma}(\bar{t})$. In others word $v(\cdot)$ is differentiable at $x_1(\bar{t_1})$ implying again that $v(\cdot)$ is differentiable at $x_1(t_1)=z$.

\end{proof}

\begin{figure}
\centering
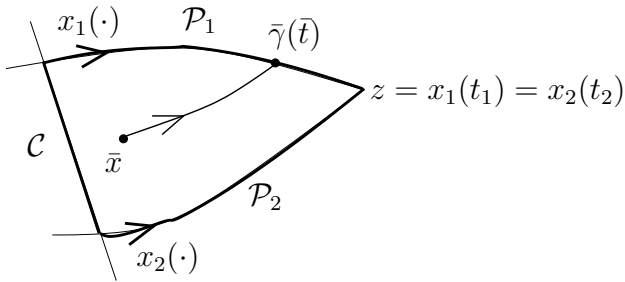
\caption{Proof of Proposition~\ref{p4}}\label{preuve1}
\end{figure}
We deduce from the above proposition the following corollary, which will be used several times in the sequel.
\begin{corollary}
\label{c2}
Assume that the hypotheses of Proposition~\ref{p4} and Condition $G$ hold true. Let $\Gamma$ be a cycle on $S$ (i.e. the support of a periodic trajectory of System~\eqref{sys0}) and let $S_1$ and $S_2$ denote the two connected components of $S\setminus\Gamma$. If $z(\cdot)$ is an extremal solution on $S$ such that $z(0)\in S_1$ then $z(t)\in S_1\cup\Gamma$ for all $t\geq0$. Moreover if $z(t_*)\in \Gamma$ for some $t_*$ then 
$z(t)\in \Gamma$ for $t\geq t_*$.
\end{corollary}
\begin{proof}
We first assume that $\{z(t):t\geq0\}\cap\Gamma=\emptyset$. Then $\{z(t):t\geq0\}\subset S_1\cup S_2$. Since $\{z(t):t\geq0\}$ is a connected subset of $S$ and $S_1,S_2$ are two open disjoints subset of $S$ then either $\{z(t):t\geq0\}\subset S_1$ or $\{z(t):t\geq0\}\subset S_2$. The second case is not possible since $z(0)\in S_1$ which implies that $z(t)\in S_1$ for all $t\geq0$.\\
We next assume that $\{z(t):t\geq0\}\cap\Gamma\neq\emptyset$. Then $z(t_0)\in\Gamma$ for some $t_0$. Denote by $\{x(t);t\in[0,T]\}$ a parametrization of $\Gamma$ with $T$ its period and by $t_{*}=\min\{t\in[0,t_0]: z(t)\in\Gamma\}$. Since $z(t_{*})=x(t_1)$ for some $t_1$ then by Proposition $\ref{p4}$ we have $z(t_{*}+t)=x(t_1+t)$ for all $t\geq0$. Therefore $z(t)\in S_1$ for $t\in[0,t^*)$ and $z(t)\in\Gamma$ for all $t\geq t_{*}$, which concludes the proof.

\end{proof}

\section{Open problems related to the geometry of Barabanov balls}\label{s3}
In this section, we present some open problems for which we provide partial answers. For simplicity of notations, 
we will deal with the Barabanov norm $v(\cdot)$ defined in~\eqref{norm}, although the results do not depend on the choice of a specific Barabanov norm.

\subsection{Uniqueness of Barabanov norms}
According to Theorem~\ref{th1}, there always exists at least one Barabanov norm. Moreover, it is clear that if $v(\cdot)$ is a Barabanov norm, then $\lambda v(\cdot)$ is a Barabanov norm as well for every positive $\lambda$. Therefore, uniqueness of Barabanov norms can only hold up to homogeneity. Our first question, for which a partial answer will be given later, is the following.  
\smallskip

\textbf{Open problem 1:}\textit{Under which conditions the Barabanov norm is unique up to homogeneity, i.e., for every Barabanov norm $w(\cdot)$ there exists $\lambda>0$ such that $w(\cdot)=\lambda v(\cdot)$ where $v(\cdot)$ is the Barabanov norm defined in~\eqref{norm}?}

\smallskip

In the following we provide some sufficient conditions for uniqueness. In order to state them,  we need to consider the union of all possible $\omega$-limit sets of extremal trajectories on $S$,
\begin{equation}
\label{eq:ome}\Omega:=\underset{\{x(\cdot):x(t)\in S\}\ }\cup\omega(x(\cdot)).
\end{equation}
\begin{theorem}
\label{pr1}
Assume that there exists a dense subset $\hat\Omega$ of $\Omega$ such that for every $z_1,z_2\in\hat\Omega$ , there exists an integer $N>0$ and extremal trajectories $x_1(\cdot),\dots,x_N(\cdot)$ with $z_1\in\omega(x_1(\cdot)),z_2\in \omega(x_N(\cdot))$ and $\omega(x_i(\cdot))\cap \omega(x_{i+1}(\cdot))\neq\emptyset$ for $i=1,\dots,N-1$.
Then the Barabanov norm is unique up to homogeneity.
\end{theorem}

\begin{proof}
{
Let $v_1(\cdot)$ and $v_2(\cdot)$ be two Barabanov norms for System $\eqref{sys0}$. Without loss of generality we identify $v_1(\cdot)$ with the Barabanov norm $v(\cdot)$ defined by~\eqref{norm}. 
We define
 \begin{align}
\bar{\lambda}:=\min\{\lambda>0: {v_1}^{-1}(1)\subset {v_2}^{-1}(\left[0,\lambda\right])\}.\label{eq-lam}
\end{align}
Standard compactness arguments show that $\bar{\lambda}$ is well defined and 
${v_1}^{-1}(1)\cap {v_2}^{-1}(\bar{\lambda})$ is non-empty. Then consider $x_0\in {v_1}^{-1}(1)\cap {v_2}^{-1}(\bar{\lambda})$ and $\hat x(\cdot)$ a $v_2$-extremal starting at $x_0$. By \eqref{eq-lam} and monotonicity of $v_1(\hat x(\cdot))$ the support of $\hat x(\cdot)$ must be included in the set ${v_1}^{-1}(1)\cap {v_2}^{-1}(\bar\lambda)$. One therefore gets that $\omega(\hat x(\cdot))\subset {v_1}^{-1}(1)\cap {v_2}^{-1}(\bar\lambda)$.
On the other hand by monotonicity of $v_2(x(\cdot))$ along any trajectory $x(\cdot)$ of the system one has that $v_2(\cdot)$ is constant on $\omega(x(\cdot))$. Thus,  under the hypotheses of Theorem~\ref{pr1}, $v_2(\cdot)$ is constant on $\Omega$ and, since $\omega(\hat x(\cdot))\subset\Omega$,  its value is $\bar{\lambda}$.
Now, given a point $x\in v_1^{-1}(1)$, one can consider a $v_1$-extremal trajectory starting from $x$ and, since the value of $v_2(\cdot)$ at the corresponding $\omega$-limit
 is $\bar\lambda$, it turns out that $v_2(x)\geq \bar\lambda$. Equation~\eqref{eq-lam} then implies that $v_2(x)=\bar\lambda$. Hence we have proved that $v_1^{-1}(1)=v_2^{-1}(\bar\lambda)$, which concludes the proof. }

\end{proof}

\noindent From the previous result one gets the following consequence.

\begin{corollary}
\label{cy}
Assume that there exists a finite set of extremal trajectories $x_1(\cdot),\cdots,x_N(\cdot)$ on $S$ such that $\Omega=\cup_{i=1,\cdots,N}{\omega(x_i(\cdot))}$ is connected. Then the Barabanov norm is unique up to homogeneity.
\end{corollary}
\begin{proof} Let us observe that for  any $\mathcal{I}\subset \{1,\dots,N\}$ different from $\emptyset$ and  $\{1,\dots,N\}$ one has $\big(\cup_{j\in \mathcal{I}}\omega(x_j(\cdot))\big)\cap \big(\cup_{j\notin \mathcal{I}}\omega(x_j(\cdot))\big)\neq\emptyset$. It is a consequence of the connectedness of $\Omega$ and the fact that each $\omega(x_j(\cdot))$ is closed. In particular there exists $i\notin\mathcal{I}$ such that $\big(\cup_{j\in \mathcal{I}}\omega(x_j(\cdot))\big)\cap\omega(x_i(\cdot))\neq\emptyset$. Let $z_1,z_2$ be arbitrary points of $\Omega$. By what precedes we can inductively construct a reordering $i_1,\dots,i_N$ of $1,\dots,N$ such that $z_1\in\omega(x_{i_1}(\cdot))$ and $\big(\cup_{j=1,\dots,k-1}\omega(x_{i_j}(\cdot))\big)\cap \omega(x_{i_{k}}(\cdot))\neq\emptyset$ for $k=2,\dots,N$. In particular, without loss of generality, $z_2\in\omega(x_{i_{k_1}}(\cdot))$ for some $k_1>1$, and there exists $k_2<k_1$ such that $\omega(x_{i_{k_2}}(\cdot))\cap \omega(x_{i_{k_1}}(\cdot))\neq\emptyset$. Thus we can inductively construct a finite sequence $k_1,\dots,k_L$ such that $k_{j+1}<k_j$, $\omega(x_{i_{k_{j+1}}}(\cdot))\cap \omega(x_{i_{k_j}}(\cdot))\neq\emptyset$ and $k_L=1$, and the assumptions of the proposition are then verified.

\end{proof} 

\begin{remark}
The assumptions of the previous result are verified for instance when the set $\Omega$ is formed by a unique limit cycle (but this is not the only case, as shown in the example below).
\end{remark}

\begin{example}
As in \cite[Example 1]{Barabanov2}, let $\M:=conv\{A_1,A_2\}$ 
\[A_1=\left(\begin{array}{ccc} 0 & 1 & 0 \\ -1 & 0 & 0 \\ 0 & 0 & -1\end{array}\right) \qquad A_2=\left(\begin{array}{ccc} 0 & 0 & 1 \\  0 & -1 & 0 \\ -1 & 0 & 0\end{array}\right).\]
The system is irreducible and stable but not asymptotically stable (it admits $\|x\|^2$ as a Lyapunov function and there are periodic trajectories). Thus the Barabanov norm in~\eqref{norm} is well defined. 
Note that the Barabanov sphere $S$ must contain the two circles \[\{(x_1,x_2,x_3)\in\mathbb{R}^3:x_1^2+x_2^2=1,\ x_3=0\}\ \mbox{and }\ \{(x_1,x_2,x_3)\in\mathbb{R}^3:x_1^2+x_3^2=1,\ x_2=0\}.\]

Let us see that the Barabanov norm is unique up to homogeneity. We claim that the $\omega$-limit set of any extremal trajectory is contained in the union of the circles defined above.
Since these circles coincide with the intersection of the sphere $S$ with the planes defined by $x_3=0$ and $x_2=0$ it is enough to show that $\min\{|x_2(t)|,|x_3(t)|\}$ converges to $0$ as $t$ goes to infinity.
Indeed, let $V(x):=\|x\|^2$. Then a simple computation leads to $\dot V(x(t)) \leq -\min\{x_2(t)^2,x_3(t)^2\}$ and, since $V$ is positive definite, $F(t):=\int_0^{t} \min\{x_2(\tau)^2,x_3(\tau)^2\}d\tau$ is a (monotone) bounded function. Since $\Vert x(t)\Vert\leq 1$ for all positive times, then $x(\cdot)$ is uniformly continuous as well as $F'(t)=\min\{x_2(t)^2,x_3(t)^2\}$ is uniformly continuous. Hence by Barbalat's lemma we have that $\lim_{t\to\infty} \min\{x_2(t)^2,x_3(t)^2\}=0$, which proves the claim. The hypotheses of Theorem~\ref{pr1} are then satisfied.
\end{example}

\begin{remark}
An adaptation of the above result can be made under the weaker assumption that the union of all possible $\omega$-limit sets of extremal trajectories on a Barabanov sphere is a union $\Omega\cup-\Omega$ where $\Omega$ is a connected set satisfying the assumptions of the theorem. 
However up to now we do not know any example satisfying this generalized assumption.
\end{remark}
At the light of the previous proposition the following variant of the {\bf Open Problem~1} is provided. 

\smallskip

\textbf{Open problem 2}: \textit{Is it possible to weaken the assumption of Theorem~\ref{pr1}, at least when $n=3$, by just asking that $\Omega$ is connected?}

\smallskip

Besides the cases studied in this section, the uniqueness of the Barabanov norm for continuous-time switched systems remains an open question. The main difficulty lies in the fact that Barabanov norms are usually difficult to compute, especially for systems of dimension larger than two.
For $n=2$ a simple example of non-uniqueness of the Barabanov norm is the following.

\begin{example}
Let $\M:=conv\{A_1,A_2,A_3\}$ with 
\[A_1\!=\!\left(\!\begin{array}{cc}-1 & 0 \\ 0 & 0\end{array}\!\right),~A_2\!=\!\left(\!\begin{array}{cc} 0 & 0 \\ 0 & -1\end{array}\!\right),~A_3\!=\!\left(\!\begin{array}{cc}-\alpha & 1 \\ -1 & -\alpha\end{array}\!\right).\]
It is easy to see that, taking  $\alpha\geq 1$,  any norm $v_{\beta}(x):=\max\{|x_1|,\beta |x_2|\}$ with $\beta\in[\frac1{\alpha},\alpha]$ is a Barabanov norm of the system. 
Moreover, one can show that the Barabanov norm defined in Eq.~\eqref{norm} is equal to $v_1(\cdot)$ and the corresponding $\omega$-limit set defined in Eq.~\eqref{eq:ome} reduces to the four points $(-1,0)$, $(0,1)$, $(1,0)$ and $(0,-1)$, which is clearly disconnected.
Note that  $v_{\beta}(\cdot)$ is a Barabanov norm even for the system corresponding to $conv\{A_1,A_2\}$, which is reducible.
\end{example}
\subsection{Supremum versus Maximum in the definition of Barabanov norm}
\label{sec-sup}

An important problem linked to the definition of the Barabanov norm given in \eqref{norm} consists in understanding under which hypotheses on $\mathcal{M}$ one has that, for every $y\in\mathbb{R}^n$, the  supremum in $\eqref{norm}$ is attained by a solution $x(\cdot)$ of System~\eqref{sys0}. If this is the case, these solutions must trivially be extremal solutions of System~\eqref{sys0} and then the analysis of the Barabanov norm defined in~\eqref{norm} would only depend on the asymptotic behaviour of extremal solutions of System~\eqref{sys0}.
As shown by the example and the discussion below, the above issue is not trivial at all and it  was actually overlooked in \cite{Barabanov3}, yielding a fatal gap in the main argument of that paper.

\begin{example}
Suppose that $\mathcal{M}:=conv\{A,B\}$ with 
\[A=\left(\begin{array}{ccc} 0 & 0 \\ 0 & -1\end{array}\right)\qquad B=\left(\begin{array}{ccc} \alpha & 3 \\   -0.6 &  0.7 \end{array}\right),\]
where $\alpha\sim 0.8896$ is chosen in such a way that $\rho(\mathcal{M})=0$. The latter condition is equivalent here to the fact that the trajectory of $B$ starting from the point $(-1,0)$ ``touches'' tangentially  the line $x_1=1$. In this case, it is easy to see that the closed curve constructed in Figure~\ref{fig1} by gluing together four trajectories of the system is the level set $V$ of a Barabanov norm. Moreover, the extremal trajectories of the system tend either to $(1,0)^T$ or to $(-1,0)^T$. 
On the other hand it is possible to construct trajectories of the system starting from $V$, turning around the origin an infinite number of times and staying arbitrarily close to $V$. One deduces that the Barabanov norm on $V$ defined in~\eqref{norm} is equal to the maximum of the Euclidean norm on $V$, which is strictly bigger than $1$.
\end{example}
Note that the matrix $A$ in the previous example is singular. It is actually possible to see, by using for instance the results of \bblue{\cite{BBM,Boscain}}, that for $n=2$, whenever  $\mathcal{M}:=conv\{A,B\}$, $A,B$ are non-singular and $\rho(\M)=0$, the supremum is always reached. This justifies the following question.

\smallskip

\textbf{Open problem 3:} \textit{Assume that $\mathcal{M}$ is made of non-singular $\mathbb{R}^{n\times n}$ matrices. Is it true that  for every $y\in\mathbb{R}^n$ the supremum in~\eqref{norm} is achieved?} 

\smallskip

\begin{figure}
\begin{center}
\includegraphics[scale=0.9]{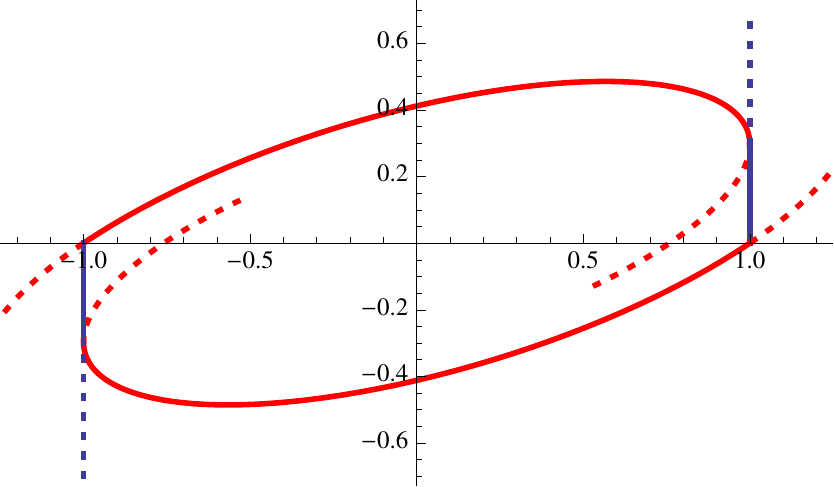}
\caption{Example where the supremum in $\eqref{norm}$ is not attained and the Barabanov norm is not strictly convex.}
\label{fig1}
\end{center}
\end{figure}

At the light of the above discussion, we can explain why the argument of \cite[Lemma 9, page 1444]{Barabanov3} presents a gap serious enough to prevent the main result of that paper to have a full valid proof. Recall first of all that \cite{Barabanov3} deals with the linear switched system {\eqref{sys0}} associated with $\mathcal{M}=conv\{A,A+bc^T\}$ where $A$ and $A+bc^T$ are $3\times 3$ Hurwitz matrices, $b,c\in \mathbb{R}^3$ such that $(A,b)$ and $(A^T,c)$ are controllable and $\rho(\mathcal{M})=0$. These assumptions imply that  $\mathcal{M}$ is irreducible. The main result of \cite{Barabanov3} states that, under the previous assumptions, there exists a central symmetric bang-bang trajectory $\gamma$ with four bang arcs (i.e. arcs corresponding to $A(t)\equiv A$ or $A(t)\equiv A+bc^T$) such that every extremal trajectory of {\eqref{sys0}} converges to $\gamma$. In the course of the argument, N. E. Barabanov cleverly introduces auxiliary semi-norms $v_m$ defined for every non zero $m\in\mathbb{R}^3$ as follows: for $x\in\mathbb{R}^3$, 
\begin{equation}\label{semi}
v_m(x)=\sup\left(\limsup_{t\rightarrow\infty} m^Tx(t)\right),
\end{equation}
where the $\sup$ is taken over all trajectories $x(\cdot)$ of {\eqref{sys0}} starting at $x$. It can be shown easily (cf. \cite[Theorem 3, page 1444]{Barabanov3}) that, for every non-zero $m\in\mathbb{R}^3$, $(i)$ $v_m$ is a semi-norm; $(ii)$ its evaluation along any trajectory of  {\eqref{sys0}} is non-increasing; $(iii)$ for every $x_0\in \mathbb{R}^3$, there exists a trajectory $x(\cdot)$ of {\eqref{sys0}} starting at $x_0$ along which $v_m$ is constant (and thus called $v_m$-extremal) as well as a trajectory $l(\cdot)$ of the adjoint system associated with $\mathcal{M}$ such that the contents of Theorem \ref{th2} hold true with $v_m$ instead of $v$ and similarly for Proposition~\ref{p3} if $v_m(x_0)>0$. 

Then comes Lemma $9$, page $1444$ in \cite{Barabanov3}. Since, on the one hand, this lemma is instrumental for the rest of Barabanov's argument  but with a gap in the proof, and the other hand, the proof given in 
\cite{Barabanov3} is rather short, we fully reproduce both the statement and the proof below.

\begin{lemma}[Lemma 9 in \cite{Barabanov3}]
\label{mar2}
Let $n=3$, $\mathcal{M}=conv\{A,A+bc^T\}$ with $b,c\in \mathbb{R}^3$ such that $(A,b)$ and $(A^T,c)$ are controllable and $\rho(\mathcal{M})=0$.
Let $m$ be a non-zero vector of $\mathbb{R}^3$. For every periodic trajectory $(x(\cdot),l(\cdot))$ of {\eqref{sys0}}, $v_m$ is constant along $x(\cdot)$ and 
\begin{equation}\label{semi1}
v_m(x(\cdot))=\max\{0,\max\ m^Tx(t),\ t\geq 0\}\}.
\end{equation}
\end{lemma}
Here is the argument of N. E. Barabanov, quoted verbatim: 

{\it Suppose that the assertion fails 
to be valid for some periodic solution $(x(\cdot),l(\cdot))$. Then $v_m(x(0))>0$ and 
$v_m(x(\cdot))>\max\{0,\max\ \{m^Tx(t),\ t\geq 0\}\}$. Let $V$ and $W$ denote the closed domains into which the unit sphere $S$ is divided by the curve $x(\cdot)$. Specifically, let $W$ 
be the domain for which $\max\{m^Ty,\ y\in W\}\geq v_m(x(0))$. Then $\max\{m^Ty,\ y\in V\}=\max\{m^Tx(t),\ t\geq 0\}<v_m(x(0))$,
since the function $v_m$ is convex. Let $(y_k)$, $k\geq 0$, be a sequence of points of $V$ where $v_m$ is differentiable and converging to $x(0)$; let $(x_k(\cdot),l_k(\cdot))$ be the extremal solutions starting at $y_k$ such that $v_m(x_k(\cdot))$ is constant. Then $v_m(y_k)$  tends to $v_m(x(0))$. Moreover, we have $l_k(t)=\nabla v_m(x_k(t))$; hence $x_k(t)\in V$ for all $t\geq 0$. Thus $\sup\{m^Tx_k(t),\ t\geq 0\}\leq \max\{m^Ty,\ y\in V\}<v_m(x(0))$.
The contradiction obtained proves the assertion.}

One can easily see that everything above is correct except the very last statement. The only way to get a contradiction is that $v_m(x_k(\cdot))=\sup\{m^Tx_k(t),\ t\geq 0\}$ along the sequence of trajectories $(x_k(\cdot))_{k\geq 0}$. It therefore implicitely assumes a positive answer to the {\bf Open problem~3} at $y_k$ (with $v$ replaced by $v_m$), maybe because there is a unique $v_m$-extremal trajectory starting at $y_k$. However, there is no argument in \cite{Barabanov3} towards that conclusion.

\subsection{Strict convexity of Barabanov balls}
\label{s5}In this section, we focus on the strict convexity of Barabanov balls. 
There is another noticeable feature of Example 1 given previously: the Barabanov unit ball (or, equivalently, the Barabanov norm) is not strictly convex since the Barabanov unit sphere contains segments. For that particular example, this comes from the fact that the matrix $A$ is singular. Hence we ask the following question, for which partial answers are collected below.

\smallskip

\textbf{Open problem 4:} \textit{
Assume that $\mathcal{M}$ is made by non-singular $\mathbb{R}^{n\times n}$ matrices.  Is it true that the Barabanov balls are strictly convex?}

\smallskip

We next provide a technical result needed to establish the main result of this section, Proposition~\ref{sc-prop1} below.
\begin{lemma}\label{sc-lem1}
Let $x_0,x_1\in S$ such that $x_0\neq x_1$ and $R(\cdot)$ be the {fundamental matrix} associated with some switching law $A(\cdot)$. Suppose that the open segment $(R(t)x_0,R(t)x_1)$ intersects $S$ for some 
 $t\geq 0$. Then the whole segment $[R(t)x_0,R(t)x_1]$ belongs to $S$.
 
 Moreover, letting $x_\beta:=\beta x_1+(1-\beta)x_0$, if $R(\cdot)x_{\alpha}$ is an extremal solution of System~\eqref{sys0} for some $\alpha\in(0,1)$ then  $R(\cdot)x_{\beta}$ is also an extremal solution for every $\beta\in [0,1]$.
\end{lemma}
\begin{proof}
For $\beta\in(0,1)$ let us define $z_\beta:=\beta R(t)x_1+(1-\beta)R(t)x_0$. If $z\in S\cap (R(t)x_0,R(t)x_1)$ then $z=z_{\alpha}$   for some $\alpha\in (0,1)$. 
By convexity of $v$ we have that $v(z_\beta)\leq1$ for each $\beta\in(0,1)$
Let us prove that $v(z_\beta)\geq1$. Since $z_\beta\in\big(R(t)x_0,R(t)x_1\big)$, then either $z\in[z_\beta,R(t)x_1)$ or $z\in(R(t)x_0,z_\beta]$. If $z\in[z_\beta,R(t)x_1)$, then 
\[
1=v(z)\leq\delta v(z_\beta)+(1-\delta)v(R(t)x_1)\leq \delta v(z_\beta)+ 1-\delta,
\]
for some $\delta\in (0,1]$ implying that $v(z_\beta)\geq1$. The second case can be treated with the same arguments. This proves the first part of the lemma. The second part is then a trivial consequence of the first one.

\end{proof}

We can now state our main result related to the strict convexity of Barabanov balls.
\begin{proposition}\label{sc-prop1}
Assume that $\mathcal{M}$ is a convex compact irreducible subset of $\mathbb{R}^{n\times n}$,  not containing singular matrices and $\rho(\mathcal{M})=0$. Then, the intersection of the Barabanov unit sphere $S$ with any hyperplane $P$ has empty relative interior in $P$.
\end{proposition}
\begin{proof}
The conclusion is clearly true if $0\in P$. So assume  that the conclusion is false for some hyperplane $P$ not containing $0$. Let  $x_0$ be a point in the interior of $S\cap P$ admitting an extremal trajectory $R(\cdot)x_0$ where $R(\cdot)$ is the {fundamental matrix} associated with some switching law $A(\cdot)$. Without loss of generality we assume that $R(\cdot)$ is differentiable at $t=0$ with $\dot R(0)=A(0)\in \mathcal{M}$. By assumption and by Lemma~\ref{sc-lem1}  for any $x\in S\cap P$ there exists a segment  in  $S\cap P$ connecting $x$ and $x_0$ and containing $x_0$ in its interior. 
Then, again by Lemma~\ref{sc-lem1},  $R(t)x\in S \cap P$ for small $t$ and for all $x$ in the interior of $S\cap P$, which
implies that $A(0)x$ is tangent to $P$, that is $l^T A(0)x=0$ where $l$ is orthogonal to $P$. But then this is also true on a cone of $\mathbb{R}^n$ with non-empty interior and thus on the whole $\mathbb{R}^n$, which is impossible since $A(0)$ is non-singular. 

\end{proof}

{The following corollary of Proposition~\ref{sc-prop1} provides an answer to the {\bf Open Problem~4} when $n=2$.}
\begin{corollary}
\label{c}
If $n=2$ and the hypotheses of Proposition~\ref{sc-prop1} hold true, then the Barabanov balls are strictly convex.  
\end{corollary}
Addressing the issue of strict convexity of Barabanov balls appears to be a complicated task in the general case. The following result shows that the Barabanov unit ball is strictly convex under some regularity condition on $\mathcal{M}$.
\begin{theorem}
\label{sc} 
If $\mathcal{M}$ is a $\mathcal{C}^1$ convex compact domain of $\mathbb{R}^{n\times n}$, irreducible and with $\rho(\mathcal{M})=0$, then Barabanov balls are strictly convex. 
\end{theorem}

\begin{proof}
By contradiction assume that there exist two distinct points $x_0, x_1\in S$ and $\lambda\in(0,1)$ such that $x_\lambda=\lambda x_1+(1-\lambda)x_0\in S$. For every $(x,l)\in S\times(\mathbb{R}^n\setminus\{0\})$, we define the linear functional $\phi_{(x,l)}(A)= l^T Ax$ on $\mathbb{R}^{n\times n}$. Notice that $\forall A\in\mathcal{M}$ there exists at most one supporting hyperplane $H$ of  $\mathcal{M}$ at $A$ given by 
$$
H=\big\{B\in\mathbb{R}^{n\times n}:\phi_{\left(\bar{x}(A),\bar{l}(A)\right)}(B)=0\big\},
$$
for some point  $(\bar{x}(A),\bar{l}(A))\in S\times(\mathbb{R}^n\setminus\{0\})$. One has also that $(\bar{x}(A),\bar{l}(A))$ is uniquely defined if it exists.
Let $x_{\lambda}(\cdot)=R(\cdot)x_{\lambda}$ be an extremal solution. Then $x_\mu(\cdot):=R(\cdot)x_{\mu}$ is extremal for all $\mu\in[0,1]$, where $x_\mu=\mu x_1+(1-\mu)x_0$. By using \eqref{c4}, one gets $x_\mu(t)=\pm \bar{x}(A(t))$ for a.e. $t$ and every $\mu\in [0,1]$ which implies that $x_0(t)=\pm x_1(t)$ for $t$ sufficiently small. Since $[x_0,x_1]\subset S$, then $x_0=x_1$. Hence the contradiction proves Theorem~\ref{sc}. 

\end{proof}

\section{A Poincar\'e-Bendixson Theorem for extremal solutions}
\label{s6}
In this section we  first show a characterization of the extremal flows in the framework of linear differential inclusions in $\mathbb{R}^n$. Then, for $n=3$, we prove a Poincar\'e-Bendixson theorem for extremal solutions of System~\eqref{sys0}. 
Similar results have been implicitly assumed and used in \cite{Barabanov3,Barabanov2}.

Recall that an absolutely continuous function is a solution of Equation~\eqref{sys0} if and only if it satisfies the differential inclusion
\[\dot x(t)\in F_{\mathcal{M}}(x),\]
where  $F_{\mathcal{M}}(x):=\{Ax: A\in\mathcal{M}\}$ is a multifunction defined on $\mathbb{R}^n$ and taking values on the power set of $\mathbb{R}^n$.
In particular it is easy to see that, for every $x\in\mathbb{R}^n$, $F_{\mathcal{M}}(x)$ is non-empty, compact and convex.
Moreover, $F_{\mathcal{M}}(\cdot)$ turns out to be upper semicontinuous on $\mathbb{R}^n$. 
Let us briefly recall, as it will be useful later, the definition of upper semicontinuity for multifunctions. We say that a multifunction  $F(\cdot)$ defined on $D\subset \mathbb{R}^k$ and taking (closed) values on the power set of $\mathbb{R}^h$ is upper semicontinuous at $x\in D$ if $\lim_{y\rightarrow x}\beta(F(y),F(x))= 0$. If $F(\cdot)$ is upper semicontinuous at every $x\in D$, we say that $F(\cdot)$ is upper semicontinuous on~$D$. Here $\beta(A,B)=\sup_{a\in A}{d(a,B)}$, where $d(a,B)$ is the Euclidean distance between the point  $a\in A$ and the set $B$. (Note that $\beta(\cdot)$ is not formally a distance since it is not symmetric and $\beta(A,B)=0$ just implies $A\subset B$.)

Extremal trajectories can also be described as solutions of a suitable differential inclusion, as we will see next. 
Let $x\in \mathbb{R}^n$ and $A\in\mathbb{R}^{n\times n}$. We say that $A$  verifies $\mathcal{P}(x)$ if there exists $l\in\partial v(x)$ such that $l^T Ax=\max_{B\in\mathcal{M}}{l^T Bx}=0$.
Consider now the linear differential inclusion
\begin{align}
\label{1}
\dot{x}\in\Hat{F}_{\mathcal{M}}(x):=\{Ax: A\in\mathcal{M}, A~\text{verifies}~\mathcal{P}(x)\}.
\end{align}
\begin{definition}
A solution $x(\cdot)$ of System $\eqref{sys0}$ is said to be solution of System $\eqref{1}$ if  it is associated with a switching law $A(\cdot)$ such that $A(t)$ satisfies $\mathcal{P}(x(t))$
for a.e. time $t$ in the domain of~$x(\cdot)$.
\end{definition}

Based on Theorem~\ref{th2}, the following result provides a necessary and sufficient condition for a solution of System~\eqref{sys0} to be extremal.
\begin{proposition}
The solutions of System~\eqref{1} coincide with the extremal solutions of System~\eqref{sys0}.
\end{proposition}
\begin{proof}
By virtue of Theorem~\ref{th2}, if $x(\cdot)$ is extremal then it is clearly a solution of System $\eqref{1}$.

On the other hand, let $x(\cdot)$ be a solution of $\eqref{1}$ associated with some switching law $A(\cdot)$ such that $A(t)$ verifies $\mathcal{P}(x(t))$ for a.e. $t$. We prove that $x(\cdot)$ is extremal. Notice that  the map $\varphi:t\mapsto v(x(t))$ is absolutely continuous on every compact interval $[a,b]\subset\mathbb{R}_+$ since $v(\cdot)$ is Lipschitz and $x(\cdot)$ is absolutely continuous. Therefore, to prove that $\varphi(\cdot)$ is constant on $[a,b]$, it is enough to show that $\varphi(\cdot)$ has zero derivative for a.e. $t\in[a,b]$. Let $t_0\in[a,b]$ be a point of differentiability of $\varphi(\cdot),\ x(\cdot)$ and such that $A(t_0)$ verifies $\mathcal{P}(x(t_0))$. Then there exists $l_0\in\partial v(x(t_0))$ such that $l_0^{T} A(t_0)x(t_0)=0$. Since 
$
\varphi(t)-\varphi(t_0)\geq l_0^T(x(t)-x(t_0))
$
for every $t\in[a,b],$
one deduces that  
$
\frac{\varphi(t)-\varphi(t_0)}{t-t_0}\geq l_0^{T}\left(\frac{x(t)-x(t_0)}{t-t_0}\right)
$
for $t>t_0$.
Passing to the limit as $t\rightarrow t_0^+$, we get $\dot{\varphi}(t_0)\geq l_0^{T}\dot{x}(t_0)=l_0^{T}A(t_0)x(t_0)=0$. On the other hand, since $v(\cdot)$ is non increasing along $x(\cdot)$ then $\dot{\varphi}(t_0)\leq 0$. Hence $\dot\varphi(t)=0$ for a.e.    $t\in[a,b]$. This proves that $x(\cdot)$ is extremal.

\end{proof}

 In the last part of this section, we focus on the asymptotic behaviour of the extremal solutions of System~\eqref{sys0}, and in particular we state a Poincar\'e-Bendixson theorem for extremal trajectories.
 From now on we will assume $n=3$, so that extremal trajectories live on a two-dimensional (Lipschitz) surface. 
\begin{remark}
The classical Poincar\'e-Bendixson results for planar differential inclusions (see e.g. \cite[Theorem~3, page~137]{Filippov}) do not apply in our case, since, besides the fact that our system is defined on a non-smooth manifold instead of $\mathbb{R}^2$, we cannot ensure that some usual requirements such as the convexity of $\Hat{F}_{\mathcal{M}}(x)$ or the upper semicontinuity of $\Hat{F}_{\mathcal{M}}(\cdot)$ are satisfied.
\end{remark}
\begin{definition}\label{def-TS}
A transverse section $\Sigma$ of $S$ for System~\eqref{1} is a connected subset of the intersection of $S$ with a plane $P$ such that for every $x\in\Sigma$ and $y\in\Hat{F}_{\mathcal{M}}(x)$ we have $w^T y>0$ where $w$ is an orthogonal vector to $P$ (see~Figure~\ref{f-trans}).
\begin{figure}
\begin{center}
\includegraphics[scale=0.4]{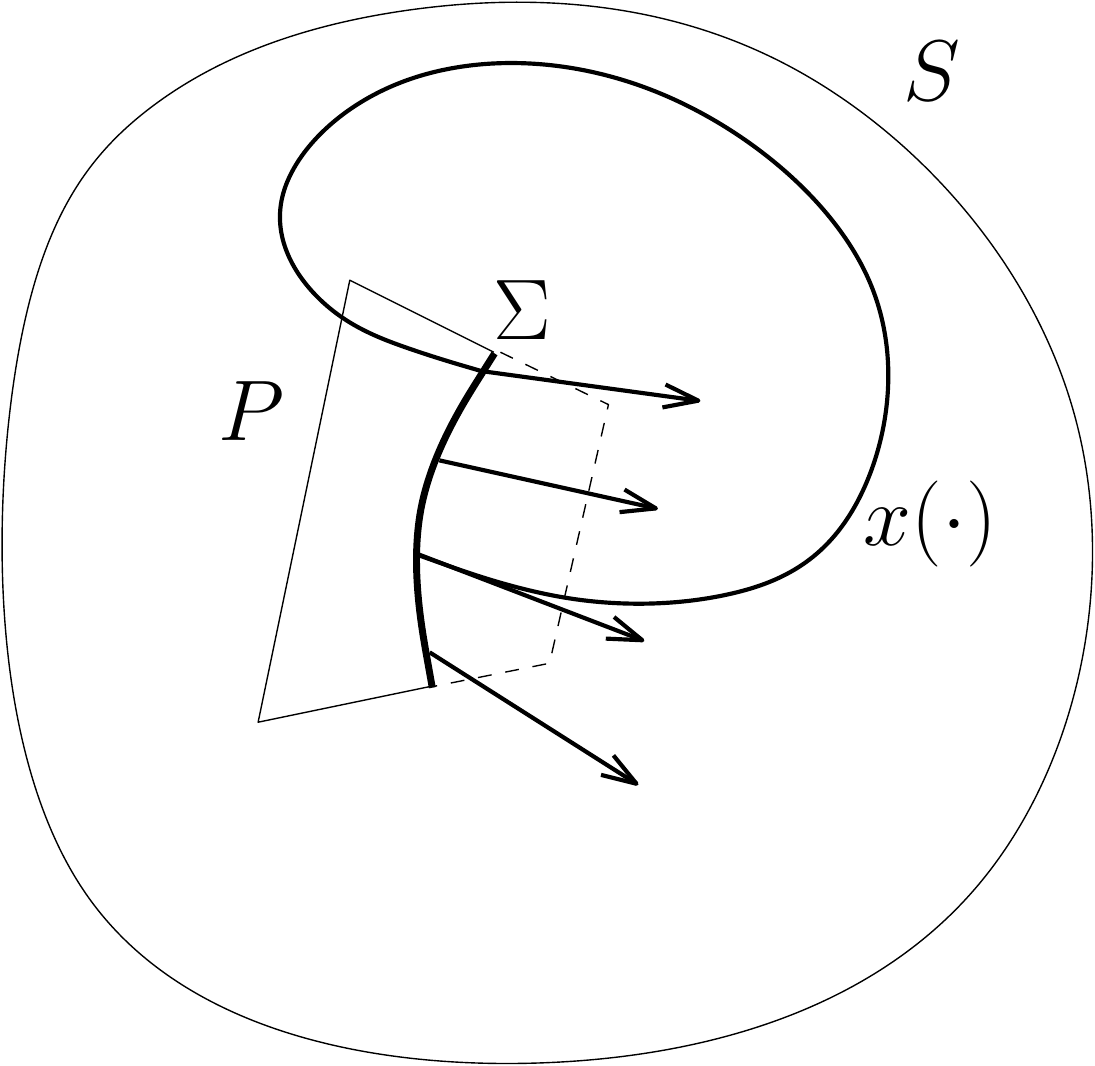}
\caption{Transverse section and proof of Poincar\'e-Bendixson theorem.}
\label{f-trans}
\end{center}
\end{figure}   
\end{definition}
\begin{definition}
We say that $z\in S$ is a stationary point of System~\eqref{sys0} if $0\in F_{\mathcal{M}}(z)$. Otherwise, we say that $z$ is a nonstationary point.
\end{definition}
\blue{The following lemma allows one to follow a strategy which is similar to the classical one in order to prove a Poincar\'e-Bendixson result.}
\begin{lemma}
\label{st}
For every nonstationary point $z\in S$, there exists a local transverse section $\Sigma$ of $S$ containing $z$.
\end{lemma}
\begin{proof}
Let $z$ be a nonstationary point of System~\eqref{sys0}. Since the set $F_{\mathcal{M}}(z)$ is convex and compact, by Hahn-Banach theorem, there exists a plane $P$ which strictly separate $0$ and $F_{\mathcal{M}}(z)$. Define
$
P:=\{x\in\mathbb{R}^3: w^T x=\nu\}
$
for some $w\in\mathbb{R}^3$ and $\nu\in\mathbb{R}$. Without loss of generality, assume that $\nu>0$. Then, the point $0$ lies in the region $\{x: w^Tx<\nu\}$, and the set $F_{\mathcal{M}}(z)$ lies in the region $\{x:w^T x>\nu\}$. Since $P$ is closed, $F_{\mathcal{M}}(z)$ is compact and $P\cap F_{\mathcal{M}}(z)=\emptyset$,  then $d_0:=\min_{y\in F_{\mathcal{M}}(z)}d(y,P)>0$. By the fact that $F_{\mathcal{M}}$ is upper semicontinous, there exists $\varepsilon_0>0$ so that if $x$ verifies $\|x-z\|<\varepsilon_0$, one has  $\beta(F_{\mathcal{M}}(x),F_{\mathcal{M}}(z)) < d_0$. It results that, for every $x\in\mathbb{R}^3$ such that $\|x-z\|<\varepsilon_0$, $F_{\mathcal{M}}(x)$ is a subset of $\{y\in\mathbb{R}^3:w^T y>\nu\}$. 
Hence, for every $x\in P\cap \{x\in S:\|x-z\|<\varepsilon_0\}$ and  $y\in\Hat{F}_{\mathcal{M}}(x)$ we have $w^T y>0$, which proves the lemma.

\end{proof}
We now state our Poincar\'e-Bendixson result for extremal solutions of System~\eqref{sys0}.
\begin{theorem}
\label{p2}
Assume that $n=3$, Condition $G$ holds true and every matrix of $\mathcal{M}$ is non-singular. Then every extremal solution of System $\eqref{sys0}$ tends to a periodic solution of~\eqref{sys0}.
\end{theorem}

If $x(\cdot)$ is a non-injective solution of  $\eqref{1}$ starting at $x_0\in S$ and associated with $A(\cdot)$, then, by Proposition $\ref{p4}$, $x(\cdot)$ is periodic on $[T,+\infty)$ for some $T\geq0$ implying that $\omega(x(\cdot))$ is a periodic trajectory, and the conclusion of Theorem~$\ref{p2}$ holds.
Thus, without loss of generality,  we will prove the theorem under the assumption  that the trajectory $x(\cdot)$ does not intersect itself, that is $x(t_1)\neq x(t_2)$ if $t_1\neq t_2$. We give the following lemma without proof because the results contained in it are standard and can be either found in $\cite{Filippov}$ or easily derived from Proposition~\ref{p4}.

\begin{lemma}
\label{3}
Given an extremal trajectory $x(\cdot)$ the following results hold true.
\begin{itemize}
\item The $\omega$-limit set $\omega(x(\cdot))$ is a compact and connected subset of $S$.
\item If $x(\cdot)$ intersects a transverse section $\Sigma$ several times, the intersection points are placed  monotonically on $\Sigma$, i.e., for any increasing sequence of positive numbers $(t_i)_{i\geq 0}$ such that $x(t_i)\in\Sigma$ for every $i\geq 0$, $x(t_i)$ belongs to the arc in $\Sigma$ between $x(t_{i-1})$ and $x(t_{i+1})$ for every $i\geq 1$.
\item The $\omega$-limit set $\omega(x(\cdot))$ of the trajectory $x(\cdot)$ can intersect a transverse section $\Sigma$ at no more than one point. Moreover, if $z$ is the point of intersection of $\omega(x(\cdot))$ and $\Sigma$ then $x(\cdot)$ intersects $\Sigma$ only at some points 
$$
z_i=x(t_i),~~~\\\\\ t_i<t_{i+1},~~\text{for}~~~\\i\geq 1,~~~\\ t_i\rightarrow+\infty,
$$
such that the sequence $(z_i)_{i\geq 1}$ tends to $z$ monotonically on $\Sigma$.
\end{itemize}
\end{lemma}

The following two lemmas are crucial in the proof of Theorem~\ref{p2}
\begin{lemma}\label{union}
The $\omega$-limit set $\omega(x(\cdot))$ of $x(\cdot)$ is a union of extremal trajectories.
\end{lemma}
\begin{proof}
Let $z\in\omega(x(\cdot))$ and $t_i\rightarrow+\infty$ a sequence of times such that $x(t_i)\rightarrow z$. By Banach-Alaoglu theorem up to subsequences $A(t_i+\cdot)\overset{w^\ast}{\rightharpoonup} \bar{A}(\cdot)\in L^\infty(\mathbb{R}_+,\mathcal{M})$.
Therefore, $x(t_i+\cdot)$ converges uniformly on compact time intervals to a solution $\bar{x}(\cdot)$ associated with $\bar{A}(\cdot)$ and such that $\bar{x}(0)=z$. Notice that since $x(\cdot)$ is extremal then passing to the limit, $\bar{x}(\cdot)$ is also extremal. This proves that $\omega(x(\cdot))$ is the union of extremal trajectories.

\end{proof}
\begin{remark}
\label{rema}
By the same argument of the previous lemma, if $z(\cdot)$ is a solution of System~$\eqref{sys0}$ and if we use $v_{z(\cdot)}$ to denote $\lim_{t\rightarrow \infty}v(z(t))$,
then the $\omega$-limit set of $z(\cdot)$  is the union of extremal solutions on $v^{-1}(v_{z(\cdot)})$.
\end{remark}
\begin{lemma}
\label{4}
For any $z\in\omega(x(\cdot))$ there exists, in $\omega(x(\cdot))$, a periodic extremal solution $x_{z}(\cdot)$ with $x_{z}(0)=z$.
\end{lemma}
\begin{proof}
Consider an extremal trajectory $x_{z}(\cdot)$ starting at $z$ as given by Lemma~\ref{union}.
Since  $\{x_{z}(t): t\geq0\}$ is a subset of the compact set $\omega(x(\cdot))$, it admits an accumulation point $\bar{z}\in\omega(x_{z}(\cdot))$. Let $\Sigma$ be a transverse section at $\bar{z}$. Note that all the points of intersection of $x_{z}(\cdot)$ with $\Sigma$ are in $\omega(x(\cdot))$. Therefore, by Lemma $\ref{3}$, all these points coincide with $\bar{z}$ so that we deduce that  $x_{z}(\cdot)$ is periodic on $[t^*,+\infty)$ for some $t^*\geq 0$. 
Since, whenever $t^*>0$ and $\varepsilon\in (0,t^*]$ is small enough, any transverse section at $x_{z}(t^*-\varepsilon)$ has a unique point of intersection with $x_{z}(\cdot)$, one easily deduces that $x_{z}(\cdot)$ is periodic on $\mathbb{R}_+$. 

\end{proof}
\textit{Proof of Theorem~\ref{p2}.} Let $x_{z}(\cdot)$ be as in Lemma~\ref{4} and $T_z$ be its period. We define 
$$
\Gamma_z:=\{x_{z}(t):t\in[0,T_z]\},~~\Gamma:=\{x(t):t\geq0\}.
$$
Without loss of generality we assume that $\Gamma\cap\Gamma_z=\emptyset$ for every $z$. Indeed if it is not the case, we get $\omega(x(\cdot))=\Gamma_z$ by Corollary~\ref{c2}. Hence Theorem~\ref{p2} is proved.

Assume now that there exists an infinite number of two by two disjoint cycles. We select among them a countable sequence $(\Gamma_{n})_{n\geq 1}$ and, for each $n$, we pick a point $x_n\in\Gamma_n$. Since $\Gamma_n\subset\omega(x(\cdot))$, then up to a subsequence one has that $x_n\rightarrow\bar{x}\in\omega(x(\cdot))$ as $n\rightarrow+\infty$. Let $\Gamma_{\bar{x}}$ be a periodic trajectory passing through $\bar{x}$ and  $\Sigma$ a local transverse section passing through $\bar{x}$. Then for $n$ large enough $\Gamma_n$ intersects $\Sigma$ at some $y_n\in\omega(x(\cdot))$. By virtue of Lemma $\ref{3}$, $\omega(x(\cdot))$ intersects $\Sigma$ at only one point which implies that for $n$ sufficiently large, $y_n=\bar{x}$. Hence $\Gamma_n=\Gamma_{\bar{x}}$ contradicting the fact that the $\Gamma_n$ are two by two disjoint.

We thus get that there exists a finite number of distinct periodic trajectories in $\omega(x(\cdot))$. Since $\omega(x(\cdot))$ is connected and the supports of these periodic trajectories  are pairwise disjoint and closed, we conclude the proof of Theorem~\ref{p2}.

$\hspace{\fill}\Box$

\section{Asymptotic properties of Barabanov linear switched systems}
\label{s-33}
In this section we analyze a special case of System~\eqref{sys0} introduced by Barabanov in~\cite{Barabanov3}. 
\subsection{Definition and statement of the main results}
\begin{definition}[\textbf{Barabanov linear switched system}]\label{Bar-LSS}
A Barabanov linear switched system is a linear switched system associated with the convex and compact set $\mathcal{M}$ of matrices 
defined as
$$
\mathcal{M}:=conv\{A,A+bc^T\}=\{A+ubc^T:u\in[0,1]\},
$$
where $A\in \R^{3\times 3}$ and $b,c\in\R^3$ verify the following: $A$ and $A+bc^T$ are Hurwitz, the pairs $(A,b)$ and $(A^T,c)$ are controllable and $\rho(\M)=0$.
Trajectories of this switched system are solutions of 
\begin{align}
\label{sys}
\dot{x}(t)= Ax(t)+u(t)bc^Tx(t),
\end{align}
where $x(\cdot)$ takes values in $\R^3$ and $u(\cdot)$ is a measurable function taking values on $[0,1]$.
 \end{definition}
Note that with the above assumptions, it is easy to see that the set $\M$ defining a Barabanov linear switched system must be irreducible. We use $v(\cdot)$ and $S$ to denote the Barabanov norm defined in Eq.~\eqref{norm} and the corresponding unit sphere, respectively.

Associated with System~\eqref{sys}, we introduce its adjoint system 
\begin{align}
\label{ad1}
\dot{l}(t)=-A^T l(t)-u(t)cb^T l(t).
\end{align}
In this section, we study the asymptotic behaviour of the extremal solutions of the above switching system. To describe our results, we need the following definitions.
\begin{definition}
\label{def-s1}
Consider an extremal trajectory $x(\cdot):\mathbb{R}_+\rightarrow \mathbb{R}^3$ and the corresponding matrix function $A(\cdot):\mathbb{R}_+\rightarrow \mathcal{M}$. 
\begin{itemize}
\item A time $t\geq 0$ is said to be  regular if there exists $\varepsilon>0$ such that $A(\cdot)$ is constant on $(t-\varepsilon,t+\varepsilon)\cap \mathbb{R}_+$. A time $t\geq 0$ which is not regular is called a switching time.
\item A switching time $t>0$ is said to be  isolated if there exists $\varepsilon>0$ such that $A(\cdot)$ is constant on both $(t-\varepsilon,t)$ and $(t,t+\varepsilon)$ (with different values for the constants in $\mathcal{M}$).  
\item A bang arc is a piece of extremal trajectory defined on a time interval where $A(\cdot)$ is constant. If all the switching times of an extremal trajectory are isolated, then the trajectory is made of bang arcs and is said to be a bang-bang trajectory.
\end{itemize}
\end{definition}
Note that these definitions can still be introduced for linear switched systems in any dimension $n\geq 2$.
 
We now state the main results of this section.  
\begin{theorem}\label{main-th}
Consider a Barabanov linear switched system defined in \eqref{sys} and $S$ its unit Barabanov sphere. The following alternative holds true:
\begin{itemize}
\item[$(a)$] either there exists on $S$ a $1$-parameter family of periodic trajectories $s\mapsto\gamma_s(\cdot)$ defined on a closed interval $[0,s_*]$ which is injective and continuous as a function with values in $\mathcal{C}^0([0,\tau])$ for any $\tau>0$ and each curve $\gamma_s(\cdot)$ has, on its period, four bang arcs for $s\in (0,s_*)$ and two bang arcs for $s\in\{0,s_*\}$;
\item[$(b)$] or there exists a finite number of periodic trajectories on $S$, each of them having  four bang arcs.
\end{itemize}
\end{theorem}
In the previous result we actually suspect that the alternative described by Item $(a)$ never occurs. 
{
\begin{proposition}
\label{omega0}
Every trajectory of a Barabanov linear switched system of the form~\eqref{sys} converges either to zero or to a periodic trajectory or to the set union of a $1$-parameter family of periodic trajectories (as described in Item $(a)$ in Theorem~\ref{main-th}).
\end{proposition}
}

\subsection{Preliminary results}
We will make use several times in the sequel of the following lemma.
\begin{lemma}\label{detB}
The set $\mathcal{M}=conv\{A,A+bc^T\}$ associated with a Barabanov linear switched system is made of non-singular matrices and $1+b^T(A^T)^{-1}c>0$.
\end{lemma}

\begin{proof}
For every $u\in[0,1]$, we denote by $\delta_u:=\det(A+ubc^T)$. Then we have 
$$
\delta_u=\det(A)\det(I_3+uA^{-1}bc^T)=(1+ub^T(A^T)^{-1}c)\det(A).
$$
It implies that $u\mapsto \delta_u$ is linear. Since both values at $u=0$ and $u=1$ are negative, $\delta_u$ must be negative as well 
for every $u\in[0,1]$, hence the conclusion.

\end{proof}

We define the following functions on $\mathbb{R}_+$:
$$
\phi_{b}(t):=b^T l(t),~~ \phi_{c}(t):=c^T x(t)~~\text{and}~~\phi(t):=\phi_{b}(t)\phi_{c}(t).
$$
The function $\phi(\cdot)$ is called the \emph{switching function} associated with the switching law $A(\cdot)$.

\begin{remark}
The maximality condition Eq.\eqref{c4} is equivalent to the following,
\begin{align}
\max_{u\in[0,1]}{u\phi(t)}=u(t)\phi(t)=-l^T(t)Ax(t)
\quad a.e.\ t\geq 0.
\end{align}
Notice that if $\phi(\bar{t})\neq0$ for some $\bar{t}\geq0$ then there exists $\epsilon>0$ such that $\phi(\cdot)$ never vanishes on $(\bar{t}-\epsilon,\bar{t}+\epsilon)\cap\mathbb{R}_+$ since $\phi(\cdot)$ is continuous. Therefore we have that  
$$
u(t)=\big(1+sgn(\phi(t))\big)/2=\big(1+sgn(\phi(\bar{t}))\big)/2, \quad \forall t\in(\bar{t}-\epsilon,\bar{t}+\epsilon)\cap\mathbb{R}_+.
$$
Hence $A(\cdot)$ is constant (equal either to $A$ or $A+bc^T$) on $(\bar{t}-\epsilon,\bar{t}+\epsilon)\cap\mathbb{R}_+$ which implies that $\bar{t}$ is a regular time of $A(\cdot)$.
\end{remark}

We then must consider the following differential inclusion 
\begin{align}
\label{sa}
\begin{array}{l}
\dot{x}(t)=Ax(t)+u(t)bc^T x(t),\\
\dot{l}(t)=-A^T l(t)-u(t)cb^T l(t),\\
u(t)\in \big(1+sgn(\phi(t))\big)/2.
\end{array}
\end{align}
We will say in the sequel that a solution $(x(\cdot),l(\cdot))$ of System~\eqref{sa} is extremal if $x(\cdot)$ is extremal in the sense of Definition~\ref{def-ext} and $l(t)\in\partial v(x(t))$ for every $t\geq0$. 
In the following we study the structure of the zeros of the switching function $\phi(\cdot)$.

\begin{lemma}
\label{lem0}
Let $\bar t\geq0$ be such that $\phi_b(\bar{t})=0$. Then $\phi_b(\cdot)$ is differentiable at $\bar t$ and we have that $\dot{\phi_b}(\bar t)=-b^T A^T l(\bar t)$. Moreover the following statements hold true:
\begin{enumerate}
\item[$(a)$] If $\dot{\phi_b}(\bar{t})\neq0$ then there exists $\epsilon>0$ such that $|\phi_b(t)|>0$ for $t\in (\bar t-\epsilon,\bar t+\epsilon)\cap\mathbb{R}_+$, $t\neq \bar t$ and $sgn\big(\phi_b|_{(\bar{t}-\epsilon,\bar{t})\cap\mathbb{R}_+}\big)=-sgn\big(\phi_b|_{(\bar{t},\bar{t}+\epsilon)}\big)$.
\item[$(b)$] If $\dot{\phi_b}(\bar{t})=0$ then $\phi_b(\cdot)$ is twice differentiable at $t=\bar t$ and $\ddot{\phi_b}(\bar{t})\neq 0$. 
In particular, there exists $\epsilon>0$ such that  $|\phi_b(t)|>0$ for $t\in(\bar t-\epsilon,\bar t+\epsilon)\cap\mathbb{R}_+$, $t\neq \bar t$ and $sgn\big(\phi_b|_{(\bar{t}-\epsilon,\bar{t})\cap\mathbb{R}_+}\big)=sgn\big(\phi_b|_{(\bar{t},\bar{t}+\epsilon)}\big)$.
\end{enumerate}
\end{lemma}

\begin{proof}
Since $l(\cdot)$ is absolutely continuous and by Equation~\eqref{sa}, $\dot{\phi_b}(t)$ is well-defined and equal to $-b^TA^T l(t)-u(t)(b^Tc) \phi_b( t)$ for  almost every $t>0$. In particular if $\phi_b(\bar t)=0$ then $\dot{\phi_b}(\bar t)=-b^TA^T l(\bar t)$ and $\dot\phi_b(t)$ keeps the same sign as  $\dot\phi_b(\bar t)$ in a neighborhood of $\bar t$, if the latter derivative is nonzero, proving Item $(a)$.

Let us prove Item $(b)$. Since $\dot{\phi_b}(\bar t)=0$, one has that $\phi_b(t)=o(t-\bar{t})$ in a neighborhood of $\bar{t}$. Therefore 
$\dot\phi_b(t)=b^T(A^T)^2l(\bar{t})(t-\bar{t})+o(t-\bar{t})$  in a neighborhood of $\bar{t}$. If $b^T(A^T)^2l(\bar{t})=0$ then the vectors $b,Ab$ and $A^2b$ would all be perpendicular to the non zero vector $l(\bar{t})$, contradicting the fact that the pair $(A,b)$ is controllable. Therefore $\dot\phi_b(t)$ is equivalent to $b^T(A^T)^2l(\bar{t})(t-\bar{t})$  in a neighborhood of $\bar{t}$, implying that $\ddot\phi_b(\bar{t})$ is defined and different from zero. If $\mu=sgn(b^T(A^T)^2l(\bar{t}))$,
then there exists $\varepsilon>0$ such that $\mu \dot{\phi_b}(s)<0$ for $s\in(\bar t-\epsilon,\bar t)$ and 
$\mu \dot{\phi_b}(s)>0$ for $s\in(\bar t,\bar t+\epsilon)$. Hence $\mu \phi_b(s)>\phi_b(\bar t)=0$ for $s\in(\bar t-\epsilon,\bar t+\epsilon)\cap\mathbb{R}_+$, $s\neq \bar t$. 

\end{proof}
The previous result tells us in particular that the zeros of the function $\phi_b(\cdot)$ are isolated. A result analogous to Lemma~\ref{lem0} holds  when one replaces the function $\phi_b(\cdot)$ by $\phi_c(\cdot)$, since the pair $(A^T,c)$ is controllable, and consequently the zeros of the function $\phi_b(\cdot)$ are also isolated.
We then have the following proposition.
\begin{proposition}
\label{pr1bis}
The zeros of the switching function $\phi(\cdot)$ are isolated and, in a neighborhood of any such zero $\bar{t}$, the sign of $\phi(\cdot)$ solely depends on the  value $(x(\bar{t}),l(\bar{t}))$. Moreover the functions $b^T l(\cdot)$ and $c^T x(\cdot)$ change sign infinitely many times.

\end{proposition}

\begin{proof}
The first part of the proposition is an immediate consequence of what precedes.

Regarding the second part, we only provide an argument for $b^Tl(\cdot)$ since the corresponding one for $c^T x(\cdot)$ is identical. Let $(x(\cdot),l(\cdot))$ be an extremal solution of System~\eqref{sa}. We assume by contradiction that the function $b^T l(\cdot)$ has a constant sign on $[T,+\infty)$ for some $T\geq0$. Without loss of generality, suppose that $b^T l(\cdot)$ is positive and let us prove that $b^T l(t)\to 0$ as $t\to+\infty$.

Define on $[T,+\infty)$ the increasing $C^1$ function $F(t):=\int_{T}^{t}{b^T l(s)ds}$ for $t\geq T$. We now claim that $\lim_{t\to+\infty}{F(t)}$ exists and is finite. Notice that we have for almost every $s$, 
\begin{align}
\label{ms}
b^T (A^T)^{-1}\dot{l}(s)=b^T (A^T)^{-1}(-A^Tl(s)-u(s)cb^Tl(s))=-b^T l(s)(1+u(s)b^T(A^T)^{-1}c).
\end{align}
By Eq.~\eqref{ms}, we have $-b^T (A^T)^{-1}\dot{l}(s)\geq\alpha b^Tl(s)$ where $\alpha:=\min\{1,1+b^T (A^T)^{-1}c\}>0$, thanks to Lemma~\ref{detB}. This implies that $$F(t)\leq -\frac{1}{\alpha}\int_{T}^{t}{b^T (A^T)^{-1}\dot{l}(s)ds}=\frac{1}{\alpha}{b^T(A^T)^{-1}(l(T)-l(t))}.$$ Hence $F(\cdot)$ is bounded since $l(\cdot)$ is uniformly bounded. This shows the claim. 

Notice that $F'(\cdot)$ is absolutely  continuous with bounded derivative. By Barbalat's lemma we conclude that $F'(t)\to 0$ as $t\to+\infty$, i.e., $b^T l(t)\to 0$ as $t\to+\infty$. 

Consider the sequence of solutions $l_n(\cdot):=l(n+\cdot)$. Then $l_n(\cdot)$ converges uniformly on every compact to some solution $l_*(\cdot)$ of System~\eqref{ad1}. Notice that $l_*(\cdot)$ is contained in $S^0$ since the sequence $(l_n(\cdot))_{n\geq0}$ is entirely contained in $S^0$. Thus, by the fact that $b^Tl(s)\to 0$ as $s\to+\infty$ we deduce that $b^T l_*(s)=0$ for every $s\geq0$. Therefore $\dot{l_*}(s)=-A^T l_*(s)$ for every $s\geq0$. Hence, since $A^T$ is Hurwitz we conclude that $\|l_*(s)\|\to +\infty$ as $s\to+\infty$  which contradicts the boundedness of $l_*(\cdot)$.

\end{proof}
As a consequence, there is no chattering phenomenom for the solutions of the linear differential inclusion defined by Eq.\eqref{sa}. Another interesting consequence of Proposition~\ref{pr1bis}  is the following generalization of 
Lemma~\ref{detB}.

\begin{corollary}\label{hurwitz}
The set $\mathcal{M}=conv\{A,A+bc^T\}$ associated with a Barabanov linear switched system is made of Hurwitz matrices.
\end{corollary}
\begin{proof}
The argument goes by contradiction. By a standard continuity argument, if the conclusion of the lemma does not hold true, then there exists $\bar{u}\in (0,1)$ such that $A_{\bar{u}}:=A+\bar{u}bc^T$
admits a purely imaginary eigenvalue, which is non zero since $A_{\bar{u}}$ is non-singular according to Lemma~\ref{detB}. Therefore, there exists $\bar{x}\in S$ such that the curve $\gamma:t\mapsto e^{tA_{\bar{u}}}\bar{x}$ defined for $t\geq 0$ is periodic. As a consequence $\gamma(\cdot)$ must be extremal since it is an admissible trajectory of System~\eqref{sys}. Moreover, there exists an adjoint trajectory $l(\cdot)$ such that $(\gamma(\cdot),l(\cdot))$ is solution of System~\eqref{sa}. Since the corresponding function $u(\cdot)$ remains constant and equal to $\bar{u}\in(0,1)$, the function $(c^T\gamma(\cdot))(b^Tl(\cdot))$ must be identically equal to zero, which is not possible by Proposition~\ref{pr1bis}. 

\end{proof} 

The following result allows us to apply Theorem~\ref{p2} to  Barabanov linear switched systems. 

\begin{proposition}
\label{bd}
The set $\mathcal{M}$ satisfies Condition $G$ introduced in Definition~\ref{condG}.
\end{proposition}

\begin{proof}
Let $x_0, l_0$ be non zero vectors in $\mathbb{R}^3$ and assume that there exist  two solutions $(x_1(\cdot),l_1(\cdot))$ and $(x_2(\cdot),l_2(\cdot))$ of Systems~\eqref{sys}-\eqref{ad1} starting at $(x_0,l_0)$, associated with the switching laws $A+u_i(\cdot)bc^T$ and satisfying $\max_{u\in[0,1]}{u\phi_i(t)}=u_i(t)\phi_i(t)$ for $i=1,2$ and a.e.~$t$, where $\phi_i(t)=(b^T l_i(t))(c^T x_i(t))$.

We prove that $(x_1(\cdot),l_1(\cdot))=(x_2(\cdot),l_2(\cdot))$. Notice that it is enough to show the latter equality on some $(0,\epsilon)$ with $\epsilon>0$ sufficiently small.

If $(b^T l_0)(c^T x_0)=0$ then by Proposition~\ref{pr1bis} there exists a small $\epsilon>0$ such that $u_i(t)=\big(1+sgn(\phi_i(t))\big)/2$ on $(0,\epsilon)$ where $sgn(\phi_i(t))$ has a common value for $i=1,2$. This proves that $(x_1(\cdot),l_1(\cdot))=(x_2(\cdot),l_2(\cdot))$ on $(0,\epsilon)$.

Assume now that $0\neq(b^T l_0)(c^T x_0)=\phi_1(0)=\phi_2(0)$. Then $u_i(t)=\big(1+sgn(\phi_i(0))\big)/2$ on $(0,\epsilon)$ for some $\epsilon>0$ sufficiently small. Hence $(x_1(\cdot),l_1(\cdot))=(x_2(\cdot),l_2(\cdot))$ on $(0,\epsilon)$.

\end{proof}

We give below a technical result (Proposition~\ref{resolv}) which turns out to be instrumental for many subsequent results of the paper.
To proceed, we start with a preliminary lemma and the following convention.
Given a vector $z\in\mathbb{R}^3$ from now on we use $P_z$ to denote the plane perpendicular to $z$.
\begin{lemma}
\label{l-tec}
Let $M=e^{t_p A_p}\dots e^{t_1 A_1}$ with an alternating choice of the $A_i$'s in $\{A,A+bc^T\}$, $A_1\neq A_p$ and $t_i>0$ for $i=1,\dots,p$.  Assume that $1$ is a double eigenvalue of $M$. Then 
either $\ker ((M^{-1})^T-I_3)=P_b$ or  $\ker (M-I_3)=P_c$ and in the latter case $e^{t_1 A_1} P_c \neq P_c$. Moreover there exists $k\in \{1,\dots ,p\}$ such that $e^{-t_k A_k^T} P_b \neq P_b$.
\end{lemma}
\begin{proof}
Since $\sup_{k\geq 0} \|M^k\|$ is finite, the Jordan blocks corresponding to the eigenvalue $1$ must be trivial. Therefore, both $\ker ((M^{-1})^T-I_3)$ and $\ker (M-I_3)$ are two-dimensional subspaces of $\R^3$. Notice next  that for every $x\in \ker (M-I_3)\cap S$ there exists a periodic trajectory
starting at $x$ and all such trajectories have the same switching law. In particular $t=0$ is a common switching time. Moreover each such periodic trajectory is a solution of \eqref{sa} with initial condition $(x,l)$ for some adjoint vector $l\in \ker ((M^{-1})^T-I_3)$.
Assume that $\ker (M-I_3) \neq P_c$. Then $\ker (M-I_3) \cap P_c = \R x_0$ and for every $x\in\ker (M-I_3)$ non collinear with $x_0$ one has $c^Tx\neq 0$. Therefore for every $x$ non collinear with $x_0$, there exists an adjoint vector $l$ associated with $x$ such that $b^Tl=0$.
If there exists $x$ non collinear with $x_0$ with two non collinear adjoint vectors as above, one gets at once that $\ker ((M^{-1})^T-I_3)=P_b$.
Otherwise one can define a map $x\mapsto l(x)$ such that $l(x)$ is the adjoint vector associated with $x$ such that $b^T l(x)=0$. It remains to show that there exist $x_1,x_2\in\ker (M-I_3)$ non collinear with $x_0$ such that the corresponding adjoint vectors  $l(x_1),l(x_2)$ are themselves non collinear. If it were not the case then all the considered $l(x)$ would be collinear with some fixed $l_*\in S^0$. Since  $S^0$ is the unit sphere of the norm $v^*$ and every $l(x)$  belongs to $S^0$, one would get that $l(x)=\pm l_*$. One would deduce that for every $x\in \ker(R(T)-I_3)\cap S$ and not collinear with $x^0$, one has that $l_*^Tx=\pm 1$ and, by continuity $l_*^Tx_0=\pm 1$. Since $\ker(R(T)-I_3)\cap S$ is connected and the closed sets $\{x\,:\, l_*^Tx=  1\}$ and $\{x\,:\, l_*^Tx=  -1\}$ are disjoint,  $\ker(R(T)-I_3)\cap S$ must be contained inside one of them, contradicting the central symmetry of $\ker(R(T)-I_3)\cap S$.

Assume now that $e^{t_1 A_1}P_c=P_c$ then $e^{k t_1 A_1}x$ belongs to $P_c\cap S$ for every $k\geq 0$ and $x\in P_c\cap S$, contradicting the fact that $A_1$ is Hurwitz.
 
 It remains to prove the last statement of the lemma. If it is not the case, then for every $k\in \{1,\dots ,p\}$ one has $e^{-t_k A_k^T} P_b = P_b$, that is $e^{-t_k A_k^T}$ restricts to an endomorphism $B_k$ on $P_b$ with $\det(B_k)>1$. 
 Moreover  $B_p\dots B_1$ is equal to the restriction of $(M^T)^{-1}$ on $P_b$ which is by assumption the identity on $P_b$. Hence a contradiction. 

\end{proof}

\begin{proposition}
\label{resolv}
Let $x^0(\cdot)$ be a periodic solution of System~\eqref{sys} with period $T>0$ and associated with some switching law $A^0(\cdot)$. Let $R(\cdot)$ the {fundamental matrix} associated with $A^0(\cdot)$. Then the eigenvalue $1$ of $R(T)$ is simple and the two other eigenvalues have modulus strictly less than $1$.
\end{proposition}

\begin{proof}
We argue by contradiction. Then either $1$ is a double eigenvalue of $R(T)$ or both $1$ and $-1$ are eigenvalues of $R(T)$. In both cases there exists a finite concatenation $M=e^{t_p A_p}\dots e^{t_1 A_1}$ with $A_i\in \{A,A+bc^T\}$ for $i=1,\dots,p$, $A_1\neq A_p$ and $t_i>0$ such that $1$ is a double eigenvalue. We apply Lemma~\ref{l-tec} and, using the last part of it, one can always assume that $\ker (M-I_3)=P_c$ and $e^{t_1 A_1} P_c \neq P_c$, up to  changing  the initial switching time. 
We first remark that $P_c\cap P_{A_1^T c}\cap S$ is made of two antipodal points $\pm x_*$ since the pair $(A^T,c)$ is controllable. Then we choose $x_*$ such that $c^TA_1^2 x_* > 0$. Observe that $P_c\cap S\setminus \{x_*,-x_*\}$ is the disjoint union of two open subsets $U^+,U^-$ in $P_c\cap S$ such that  $c^TA_1 x > 0$ on $U^+$ and $c^TA_1 x < 0$ on $U^-$. Consider the trajectories $t\mapsto e^{t A_1} x$ where $x\in U^-$ is close enough to $x_*$ in such a way that $c^TA_1^2 x > 0$.

We claim that there exists a small enough open neighborhood $V$ of $x_*$ such that for every $y\in V^-:=V\cap U^-$ there exists $t(y)>0$ and $x(y)\in U^+$ such that $x(y) = e^{t(y) A_1} y$ and $t(\cdot)$ is continuous with $\lim_{y\to x_1} t(y)=0$. The claim is an immediate consequence of the implicit function theorem applied at the point $(0,x_*)$ to the function 
$F:\R\times P_c\to \R$ defined by
\begin{figure}
\begin{center}
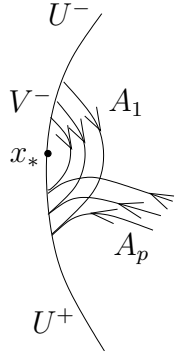
\caption{Proof of Proposition~\ref{resolv}}\label{preuve0}
\end{center}
\end{figure}
$$F(t,x)=\left\{\begin{array}{ll} c^T\frac{e^{t A_1}-Id}{t}x & t\neq 0,\\
c^T A_1 x & t=0,
 \end{array}\right.$$
 and the fact that $e^{tA_1}y$ stays on $S$ for $t\in[0,t_1]$. 
By definition of $M$ we deduce that  $e^{-t_p A_p}(P_c\cap S)\subset S$ and then every point $x(y)$ with   $y\in V^-$ is reached by an extremal trajectory  $e^{t_p A_p}z$ for some $z\in S$. Since $A_p\neq A_1$ every point $x(y)$ with   $y\in V^-$ is reached by two extremal trajectories corresponding to $A$ and $A+bc^T$ (see Figure~\ref{preuve0}). This also holds for points of the type $e^{-t A_1} x(y)$ with $t\in (0,t(y))$ small enough and $y\in V^-$. We have then reached a contradiction with Proposition~\ref{p4}.

\end{proof}

\begin{proposition}\label{bmba}
Every extremal solution $(x(\cdot),l(\cdot))$ of System~\eqref{sa} converges to a periodic solution $(\bar{x}(\cdot),\bar l(\cdot))$ of System~\eqref{sa}, where $\bar{x}(\cdot)$ and $\bar l(\cdot)$ have the same minimal period $T$ and the corresponding switching law is of minimal period $T/2$ or $T$.
\end{proposition}

\begin{proof}
Let $(x(\cdot),l(\cdot))$ be defined as above and associated with a switching law $A(\cdot)$.  By Theorem~\ref{p2}, $x(\cdot)$ tends to a periodic solution $\bar x(\cdot)$ of System~\eqref{sys} with minimal period $T>0$ and associated with some switching law $\bar A(\cdot)$ such that $A(\cdot)\rightharpoonup \bar A(\cdot)$ in the weak-$*$ topology. 
One first deduces that $ \bar A(\cdot)$ is periodic of period $T$ since $\dot {\bar x}(\cdot)$ and the zeros of $c^T\bar x(\cdot)$ are isolated.
We next claim that the minimal period of $ \bar A(\cdot)$ is $T/2$ or $T$. To see that, first recall that the set of positive periods of $ \bar A(\cdot)$
is a subgroup $G$ of $(\mathbb{R},+)$ and not reduced to zero since $T\in G$. Therefore either $G=\mathbb{R}$ or it is generated by some positive $\bar T$ which is the minimal period of $ \bar A(\cdot)$ and $T=p\bar T$ with $p$ positive integer.
In the first case, $ \bar A(\cdot)$ must be constant and, according to Corollary~\ref{hurwitz}, this contradicts the fact that $\bar x(\cdot)$ is periodic. In the second case, one has $R(T)=R(\bar T)^p$ where $R(\cdot)$ is the fundamental matrix associated with $\bar A(\cdot)$. 
Let $\omega$ be an eigenvalue of $R(\bar T)$ of modulus equal to one. 
According to Proposition~\ref{resolv}, $1$ is a simple eigenvalue of $R(T)$. Then $\omega$ must be real, equal to $1$ or $-1$. Recalling now that $T$ is the minimal period of $\bar x(\cdot)$, we get at once that $\bar T=T$ or $T/2$.

The trajectory $l(\cdot)$ belongs to $S^0$ which is compact. Therefore its $\omega-$limit set is made of solutions of System~\eqref{ad1} and associated with $\bar A(\cdot)$. Proving  the proposition amounts to show that this  $\omega$-limit set reduces to a unique periodic solution  $\bar l(\cdot)$ of System~\eqref{ad1} with minimal period $T$. 

Setting $x_0=\bar x(0),l_0=\bar l(0)$ and $S(\cdot)=\big(R(\cdot)^T\big)^{-1}$ we have that $R(T)x_0=x_0$ and that $1$ is a simple eigenvalue of $S(T)$, while the other eigenvalues have modulus strictly larger than $1$. 
Let $l_1$ be an eigenvector of $S(T)$ associated with $1$. Then $l_0=\alpha l_1+w$ where $w\in \ker(S(T)-I_3)^\perp$ and $S(nT)l_0=S(nT)\alpha l_1+S(nT)w=\alpha l_1+S(T)^n w$ for every integer $n$. Since the sequence  $(S(nT) l_0)_{n\geq 1}$ is uniformly bounded we deduce that $w=0$ and $\bar l(\cdot)$ is periodic of period $T$. If  $\bar T=T$, we are finished. Otherwise we have that 
$S(\bar T)l_0$ is either equal to $l_0$ or $-l_0$. In the first case, $l_0$ belongs to both $\partial v(x_0)$ and $\partial v(-x_0)$, which is a contradiction. The proof of the proposition is complete.

\end{proof}

\begin{remark}\label{rem1} 
Assume that there exists a $T$-periodic trajectory 
$x(\cdot)$ with support $\Gamma$ such that $\Gamma\cap (-\Gamma)\neq\emptyset$. Then, necessarily $\Gamma= -\Gamma$, $x(\cdot)$ is $\frac{T}2$-antiperiodic and such a trajectory must be unique. 
\end{remark}

In the following, we use $w(\cdot)$ to denote a Barabanov norm (i.e., a norm on $\mathbb{R}^3$ satisfying Conditions~$1$ and $2$ of Theorem~\ref{th1}) for the following linear switched system 
\begin{align}
\label{snr}
\dot{m}(t)=A^T m(t)+u(t)cb^T m(t),
\end{align}
where $u(\cdot)$ is a measurable function taking values on $[0,1]$.

\begin{lemma}
\label{int}
The trajectory $\bar{l}(\cdot)$ defined in Proposition~\ref{bmba} cannot intersect itself.
\end{lemma}

\begin{proof}
Consider the trajectory $\bar{m}(\cdot)$ of System~\eqref{snr} defined as $\bar{m}(t)=\bar{l}(T-t)$ for every $t\in[0,T]$. Then $w(\bar{m}(t))=w(\bar{m}(0))$ for every $t\in[0,T]$ since $w(\cdot)$ is a Barabanov norm for System~\eqref{snr} and $\bar{l}(\cdot)$ is $T$-periodic. Since $\rho(\mathcal{M}^T)=\rho(\mathcal{M})=0$ and by using $w(\cdot)$ instead of $v(\cdot)$ in Proposition~\ref{p4}, we deduce that $\bar{l}(\cdot)$ cannot intersect itself. 

\end{proof}

For $t\geq0$, we use ${\dot{x}}_{+}(t)$, ${\dot{x}}_{-}(t)$, respectively to denote the right and the left derivative of $x(\cdot)$ at $t\geq0$ respectively. Similarly, we use ${\dot{l}}_{+}(t)$ and ${\dot{l}}_{-}(t)$ respectively to denote the right and the left derivative of $l(\cdot)$ at $t\geq0$ respectively. The latter are well-defined since $(x(\cdot),l(\cdot))$ is a bang-bang trajectory. In that context, the definition of tranverse section to trajectories $x(\cdot)$ 
as given in Definition~\ref{def-TS} is equivalent to the following one, which also extends to the component $l(\cdot)$:
a transverse section $\Sigma$ in $S$ ($S^0$ respectively) for System~\eqref{sa}  is a connected subset of the intersection of $S$ ($S^0$ respectively) with a plane $P$ such that for every trajectory $(x(\cdot), l(\cdot))$ of System~\eqref{sa} and $t>0$ satisfying $x(t)\in\Sigma$ ($l(t)\in\Sigma$ respectively) we have $q^T {\dot{x}}_{+}(t)>0$ and $q^T {\dot{x}}_{-}(t)>0$ ($v^T {\dot{l}}_{+}(t)>0$ and $v^T {\dot{l}}_{-}(t)>0$ respectively) where $q$ ($v$ respectively) is a given vector orthogonal to $P$.

We next provide the counterpart of Lemma~\ref{3} when replacing $x(\cdot)$ by $l(\cdot)$. 
\begin{lemma}
\label{ct}
Let $(x(\cdot),l(\cdot))$ be an extremal solution of System~\eqref{sa} and $\Sigma$ a transverse section in $S^0$ for System~\eqref{sa}. If $l(\cdot)$ does not intersect itself then
$\omega(l(\cdot))$ intersects $\Sigma$ at no more than one point.
\end{lemma}

We now provide a crucial property regarding the switching function $\phi(\cdot)$.

\begin{proposition}
\label{gn}
Under the assumptions and notations of Proposition~\ref{bmba}, the functions $c^T \bar{x}(\cdot)$ and $b^T \bar{l}(\cdot)$ change sign exactly twice on each minimal period. 
\end{proposition}

\begin{proof}
We only deal with the function $b^T  \bar{l}(\cdot)$. By the same arguments the conclusion can also be obtained for the function $c^T \bar{x}(\cdot)$. 

Notice first by virtue of Lemma~\ref{lem0} that the function $b^T\bar{l}(\cdot)$ changes sign at some time $\bar t>0$ if and only if $b^T \bar{l}(\bar t)=0$ and $b^T\dot{\bar{l}}(\bar t)\neq0$. By Proposition~\ref{pr1bis} the function $b^T\bar{l}(\cdot)$ changes sign infinitely many times. Then with no loss of generality, we assume that $b^T \bar{l}(0)=0$ and $b^T\dot{\bar{l}}(0)>0$ which  implies that $b^T \bar{l}(t)>0$ for $t$ sufficiently small.  Define now the set
$
\Sigma_-:=P_b\cap\{l\in S^0:b^T A^T l<0\}.
$
We claim that $\Sigma_-$ is a transverse section passing through $\bar{l}(0)$. Indeed observe first that $\bar{l}(0)\in\Sigma_-$ since $b^T\dot{\bar{l}}(0)=-b^T A^T \bar{l}(0)$.  
On the other hand any trajectory $\tilde{l}(\cdot)$ such that $\tilde{l}(\sigma)\in\Sigma_-$ for some $\sigma\geq0$ satisfies $b^T \tilde{l}(\sigma)=0$ and $b^T\dot{\tilde{l}}(\sigma)=-b^T A^T \tilde{l}(\sigma)>0$, which proves the claim. By applying Lemmas~\ref{int} and~\ref{ct}, we conclude that $\{\bar{l}(t):t\in[0,T]\}\cap\Sigma_-=\{\bar{l}(0)\}$.
Also by Proposition~\ref{pr1bis} there exists $t^*\in(0,T)$ such that $b^T \bar{l}(t^*)=0$ and $b^T\dot{\bar{l}}(t^*)<0$. Otherwise the function $b^T \bar{l}(\cdot)$  does not change sign. Thus, we have that $b^T A^T \bar{l}(t^*)>0$. 
By defining the set $\Sigma_+:=P_b\cap\{l\in S^0:b^T A^T l>0\}$ 
and by using the same techniques as what precedes, one can show that $\Sigma_+$ is a transverse section passing through $\bar{l}(t^*)$. Thus, we have that $\{\bar{l}(t):t\in[0,T]\}\cap\Sigma_+=\{\bar{l}(t^*)\}$. Hence, the function $b^T \bar{l}(\cdot)$ changes sign only twice on each minimal period. 

By the same techniques, one can prove that the function $c^T \bar x(\cdot)$ also changes sign exactly twice on each minimal period. 

\end{proof}
We deduce from the previous results the following.
\begin{corollary}\label{cor1}
 {Every periodic trajectory $(x(\cdot),l(\cdot))$ of \eqref{sa} of minimal period $T>0$ is formed either by two bang arcs or by four bang arcs. The first case happens whenever $b^Tl(\cdot)$ and $c^Tx(\cdot)$ have common zeroes.} 
\end{corollary}
We close this section with a technical result which will be repeatedly used in the final part of the paper.
\begin{lemma}\label{lem!!}
 Let $(x_*,-x_*)$ be the unique pair of antipodal points on $S$ such that $c^Tx_*=c^TAx_*=0$. Then, there does not exist a periodic trajectory of System \eqref{sys} passing through $x_*$ or $-x_*$.
\end{lemma}
\begin{figure}
\begin{center}
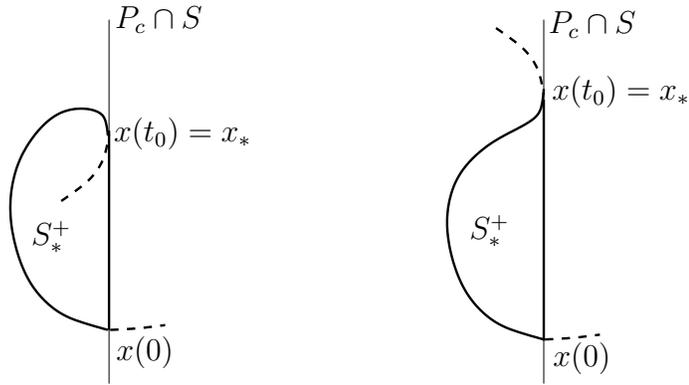
\caption{Proof of Lemma~\ref{lem!!}}
\label{fig_lemma}
\end{center}
\end{figure}

\begin{proof} Consider  $x_*$ so that $c^Tx_*=c^TAx_*=0$ and $c^TA^2x_*>0$. The argument of the lemma goes by contradiction. Let then $x(\cdot)$ be a periodic trajectory of System \eqref{sys} with $x(t_0)=x_*$ for some $t_0$ (the case  $x(t_0)=-x_*$ is treated similarly). With no loss of generality, we assume that $c^Tx(0)=0$ and $c^Tx(t)>0$ for  $t\in (0,t_0)$. Indeed, $c^T\dot x(t_0)=c^TAx_*=0$ and $c^Tx(\cdot)$ is twice differentiable at $t=t_0$ with second derivative equal to $c^TA^2x_*$. 

 Consider the closed curve $\gamma$ on $S$ defined as the union of the trajectory $x(t)$, $t\in [0,t_0]$ and the arc $\xi\subset S\cap P_c$ connecting $x(0)$ and $x(t_0)$ and not containing $-x_*$. Note that $\gamma$ is a Jordan curve on $S$ dividing $S\setminus\gamma$ into two open and disjoint subsets $S_*^{\pm}$, with $S_*^-$ containing $-x_*$. 
Moreover, $c^TAx\geq 0$ for $x$ in the arc $\xi$ since $c^TAx(0)=c^T\dot x(0)>0$,  $c^TAx_*=0$ and $-x_*$ does not belong to $\xi$. 
In addition any regular parameterisation of $S\cap P_c$ in a neighborhood of $x_*$ results into a Lipschitz curve differentiable at $x_*$ with a tangent vector at $x_*$ which is collinear with $Ax_*$ (see Figure~\ref{fig_lemma}). 
Two possibilities may occur: either $Ax_*$ points in the direction of $\xi$ or not. If it does, the piece of trajectory $x(t)$ for $t>t_0$ close to $t_0$ is contained in $S_*^+$. Then, $S_*^{+}$ must be a positive time invariant set for the dynamics defined by \eqref{sys} since $c^TAx\geq 0$ for $x$ on the arc $\xi$. However, the invariance of $S_*^{+}$  implies that $c^Tx(t)$ keeps a constant sign for every positive $t$, which contradicts Proposition~\ref{pr1bis}. 

Assume now that $Ax_*$ does not point in the direction of $\xi$. Consider the regular parameterization $\eta(\cdot)$ of $S\cap P_c$ in a neighborhood of $x_*$ given by
$$
\eta(s)=\frac{x_*+Ax_*(s-t_0)}{v(x_*+Ax_*(s-t_0))},
$$
for $s$ close to $t_0$. According to the assumption on $Ax_*$, the arc $\xi$ (in a neighborhood of $x_*$) corresponds to values $\eta(s)$ for $s$ smaller than $t_0$, and by a direct evaluation, $c^TA\eta(s)<0$ there, which contradicts the fact that  $c^TAx\geq 0$ for $x$ in the arc $\xi$.

\end{proof}

\begin{remark}\label{remL} The conclusion of the above lemma holds with $S$, $c,x_*$ and System~\eqref{sys} replaced by $S^0$, $b,l_*$
and System~\eqref{ad1}, where $(l_*,-l_*)$ is the  the unique pair of antipodal points on $S^0$ such that $l_*^Tb=l_*^TAb=0$. 
\end{remark}

\subsection{Proofs of Theorem~\ref{main-th} {and Proposition~\ref{omega0}}}
We start this section by proving Theorem~\ref{main-th}. The argument proceeds by considering the alternative of having or not an infinite number of distinct periodic trajectories on the unit Barabanov sphere. 
 
Note that Corollary \ref{cor1} provides an algebraic criterion for an extremal trajectory of System~\eqref{sys} to be periodic. 
Indeed, given such a trajectory $\gamma(\cdot)$  there exist $(t^*_1, t^*_2, t^*_3, t^*_4)\in\R_+^4$ and $x^*\in P_c\cap S$ such that $\gamma(0)=x^*$, the $t_i^*$'s are the time durations between the switchings  (the first switching time being at time $t=0$) and $\gamma(\cdot)$ is periodic of period $\sum_{i=1}^4 t^*_i$.
Then $x^*$ is an eigenvector associated with the simple eigenvalue $1$ for the matrix 
$$
M(t^*_1, t^*_2, t^*_3, t^*_4)= e^{t^*_4A_2}e^{t^*_3A_1}e^{t^*_2A_2}e^{t^*_1A_1}.
 $$
We can reformulate the above by introducing the functions $M:\R^4\rightarrow \mathbb{R}^{3\times 3}$ and $f:\R^4\rightarrow \R$  defined by
\beq\label{eq-f-0}
M(t_1,t_2,t_3,t_4)=e^{t_4A_2}e^{t_3A_1}e^{t_2A_2}e^{t_1A_1},\quad      f(t_1,t_2,t_3,t_4)=\det(M(t_1,t_2,t_3,t_4)-I_3).
\eeq
It is trivial to see that $M$ and $f$ are real analytic functions and we use $Z_f$ to denote the zero set of $f$. Setting $z^*:=(t^*_1, t^*_2, t^*_3, t^*_4)$, one gets that $z^*\in Z_f$. 

 We next provide the following lemma.  
 
 \begin{lemma}\label{lem:accu}
Assume there exists  an infinite sequence $(\gamma_n)_{n\geq 0}$ of distinct periodic extremal trajectories. Then there exists  a non trivial  interval $I\supset [0,s_*]$ and a non constant continuous  injective curve admitting a piecewise analytic parameterization  $z:I\rightarrow \R^4$
such that, 
\begin{description}
\item[$(i)$] for every $s\in I$, $f(z(s))=0$;
\item[$(ii)$] for every $s\in I\cap [0,s_*]$, $z(s)\in (\R_+)^4$; 
\item[$(iii)$] for $s\in \{0,s_*\}$, there exists $1\leq i\leq 4$ such that $t_i(\cdot)$ changes sign at $s$ while  $t_j(s)>0$ for some $j\neq i$.
 \end{description}
\end{lemma}
 \begin{proof} For every $n\geq 0$, let $z_n=(t_n^{(1)}, t_n^{(2)}, t_n^{(3)}, t_n^{(4)})^T\in \R_+^4$ be the $4$-tuple of switching times corresponding to $\gamma_n$ and $M_n:=M(z_n)$. In addition, if $x_n$ is an eigenvector of $M_n$ associated with the eigenvalue $1$ for $n\geq 0$, we can assume  that $c^TAx_n>0$ according to Lemma~\ref{lem!!}.  In that case,  for every $n\geq 0$, $x_n$ is the unique vector of $P_c\cap S$ such that $M_nx_n=x_n$. Moreover,  the points $z_n$, $n\geq 0$, are two by two distinct. Indeed, one would have otherwise that  $M_{n_1}=M_{n_2}$ for two distinct indices $n_1$ and $n_2$, which would imply that $x_{n_1}=x_{n_2}$ since $1$ is a simple eigenvalue of the $M_n$'s and leading to $\gamma_{n_1}=\gamma_{n_2}$.
 
Due to the fact that the matrices of the form $A+ubc^T$, $0\leq u\leq 1$, are Hurwitz (cf. Corollary~\ref{hurwitz}), the times $t_n^i$ are uniformly bounded for $n\geq 0$ and $1\leq i\leq 4$, yielding  that there exists a compact subset $K\subset \R_+^4$ such that $z_n\in K$ for every $n\geq 0$, the set $Z_f\cap  (\R_+)^4$ is compact and $0\in Z_f$ is isolated in $Z_f\cap  (\R_+)^4$. Up to a subsequence, we also have that the sequence 
$(z_n)_{n\geq 0}$ converges to some $z_*\in Z_f\cap  (\R_+)^4$. For later use, we consider $Z_f^+$ the set $Z_f\cap  (\R_+)^4$ minus its isolated points.  

We first prove that $Z_f^+$ does not contain any non trivial Jordan curve.
 Reasoning by contradiction, there exists a non trivial curve 
$\Gamma:S^1\rightarrow Z_f^+$ such that one can associate with any of its point  $\Gamma(\theta)$ an eigenvector $x(\theta)$ of $M(\Gamma(\theta))$ in $P_c\cap S$. The curve $x:S^1\rightarrow P_c\cap S$ can be chosen to be continuous (except possibly at a single point) and it is clearly injective, since otherwise there would exist distinct periodic trajectories starting from the same point. It implies that $\theta\mapsto c^TAx(\theta)$ keeps a constant sign according to Lemma~\ref{lem!!}.
 Moreover, the closure of the support of $x$ must contain an arc joining a pair of antipodal points and thus  there exists $\bar{\theta}\in S^1$ so that $x(\bar{\theta})=\pm x_*$, contradicting Lemma~\ref{lem!!}. 
 
 As a consequence of the Lojasiewicz's structure theorem for varieties (\cite[Theorem~6.3.3, page 168]{Krantz}) and the fact that $z_*\in Z_f^+$, there must exist an analytic submanifold in $Z_f\setminus\{0\}$ of dimension $m$ with $1\leq m\leq 3$ whose closure contains $z_*$. 
The non existence of a  non trivial Jordan curve in $Z_f^+$ implies at once that $Z_f^+$ does not contain any stratum of dimension $m\geq 2$ and thus $Z_f^+$ only contains strata of dimension zero and one. 
In that case, by using the theorem of resolution of singularities (see for instance \cite{BM}),
the conclusion of Lojasiewicz's structure theorem can be strengthened as follows: every stratum of dimension zero in $Z_f^+$ is in the interior of a continuous injective curve contained in $Z_f$ and admitting an analytic  parameterization.
One further notices that there must exists only one such curve going through any stratum of dimension zero $\bar{z}$ otherwise, if $\bar{x}\in P_c\cap S$ denotes the eigenvector of $M(\bar{z})$ associated with the eigenvalue $1$ with $c^TA\bar{x}>0$, there would exists a neighborhood $\bar{X}$ of $\bar{x}$ open in $P_c\cap S$ such that every $y\in \bar{X}\cap(P_c\cap S)$ is the eigenvector associated with $1$ of at least two distinct values of $M(\cdot)$, which is impossible. From this fact we also deduce that $Z_f^+\cap (\mathbb{R}_+^*)^4$ has a structure of a topological one-dimensional manifold (embedded in $(\mathbb{R}_+^*)^4$), and by a simple adaptation of the previous arguments $Z_f^+$ is itself a compact  one-dimensional manifold (with or without boundary). By applying standard results (see e.g. Exercises 1.2.6 and 1.4.9 in~\cite{Hirsch}) we deduce that each connected component of $Z_f^+$ is either homeomorphic to a circle or to a closed segment. The first possibility has been previously ruled out and since the arguments above imply that the extremities of each connected component belong to the boundary of  $(\mathbb{R}_+)^4$, we get the thesis.

 \end{proof}
 
Theorem~\ref{main-th} easily follows from the previous lemma.
Indeed, assume that Item $(b)$ of Theorem~\ref{main-th} does not hold, i.e. there exists an infinite number of periodic trajectories for System~\eqref{sys}. Then, Item $(iii)$ of Lemma~\ref{lem:accu} immediately implies that there exists a $1$-parameter family of periodic trajectories $s\mapsto \gamma_s(\cdot)$ with at most four bang arcs bounded at  its extremities by periodic trajectories with exactly two bang arcs. Moreover the mapping $s\mapsto \gamma_s(\cdot)$ is injective and  continuous as a function with values in $\mathcal{C}^0([0,T])$ for $T>0$, which concludes the proof of  Item $(a)$ of Theorem~\ref{main-th}.  
The theorem is proved.  
 
\medskip

We now prove Proposition~\ref{omega0}. We start the argument by making the following remark. 

\begin{remark}\label{order}
{By the same reasoning which proved that the accumulation of periodic trajectories yields the existence of a continuum of periodic trajectories (i.e. Item $(a)$ of Theorem~\ref{main-th}), one can show that there exists only a finite number of such continua of periodic trajectories on the unit Barabanov sphere. Thus there exists a finite number of isolated periodic trajectories and isolated continua of periodic trajectories $(\Gamma_i)_{i=1,\dots,p}$. 
It is useful to define an order on the set  $(\Gamma_i)_{i=1,\dots,p}$. For this purpose let us call $\hat{\Gamma}_i$ the intersection of $\Gamma_i$ with the arc $\Lambda_c^+=\{x\in P_c\cap S:c^TAx>0\}$ ($\hat{\Gamma}_i$ reduces to a point if $\Gamma_i$ is a single periodic trajectory); if $x_*$ is the point of $P_c\cap S$ satisfying $c^TAx_*=0$ and $c^TA^2x_*>0$ then we say that  $\Gamma_i<\Gamma_j$ if $\hat{\Gamma}_i$ is contained in the interior of the arc of $\Lambda_c^+$ joining $\hat{\Gamma}_j$ and $x_*$.
Equivalently $\Gamma_i<\Gamma_j$ if $\Gamma_i$ is contained in the connected subset of $S$ enclosed by $\Gamma_j$ which contains $x_*$.
}
\end{remark}

We will need the following lemma.
\begin{lemma}\label{stability0}
 {Consider $\Gamma$  as defined previously.  It divides $S\setminus  \Gamma$ into two connected components  whose closures will be denoted by $S_1$ and $S_2$. Then, for $i=1,2$, there exists an open neighborhood $U_i$  of $\Gamma$ such that either all extremal trajectories starting in $S_i\cap U_i$  converge to $\Gamma$ or no one does.}  

\end{lemma}
\begin{proof} 
To prove the lemma we will show that if there exists a single trajectory $x(\cdot)$ converging to $\Gamma$ from one of the two corresponding connected components, say $S_1$, then it is possible to define an open neighborhood $U_1$  of $\Gamma$ in $S_1$  such that all extremal trajectories starting in $U_1$ converges to $\Gamma$.
According to Lemma \ref{lem!!}, if $\bar x\in P_c\cap \Gamma$ then there exists a neighborhood $O$ of $\bar x$ in $S$ such that $P_c\cap O$ is a transverse section for System~\eqref{sa}. In particular, by Lemma~\ref{3}, $x(\cdot)$ intersects $P_c\cap O$ at a sequence of points $(x(t_k))$, $k\geq 0$, such that $t_k<t_{k+1}$ and $\lim_{l\rightarrow\infty}x(t_k)=\bar{x}$ monotonically.
Consider the open set $U$ of $S_1$ whose boundary is equal to the union of $\partial S_1$, of $\{x(t),\ t_0\leq t\leq t_{1}\}$ and the open arc  in $P_c\cap O$ connecting $x(t_0)$ and $x(t_{1})$, denoted by $(x(t_0),x(t_{1}))$. 
Then, following the classical arguments in the proof of Poincar\'e-Bendixson Theorem, together with Proposition~\ref{p4} and Corollary~\ref{c2}, it turns out that $U$ is invariant and it does not contain periodic trajectories. Thus every extremal trajectory starting in $U$ converges to $\Gamma$, concluding the proof.

\end{proof}

\begin{figure}
\begin{center}
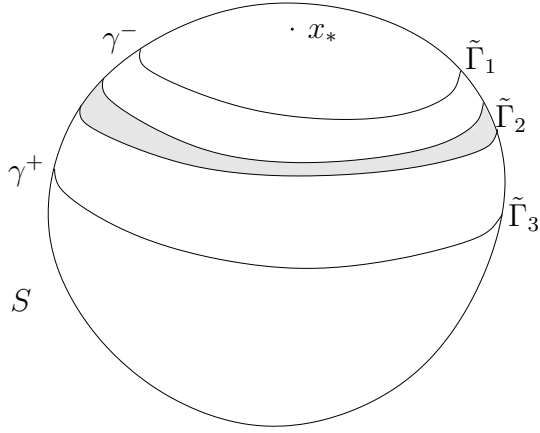
\caption{Periodic trajectories in the $\omega$-limit set $\omega(x(\cdot))$.}
\label{f-fnotations}
\end{center}
\end{figure}

Let $x(\cdot)$ be a non trivial trajectory of  System~\eqref{sys} which does not converge to zero as $t$ tends to infinity. Then, as $v(x(\cdot))$ is non-increasing, $V:=\lim_{t\rightarrow\infty}v(x(t))$ exists 
and is strictly positive. Assuming without loss of generality $V=1$ we deduce that the $\omega$-limit set $\omega(x(\cdot))$ of $x(\cdot)$ is included in $S$. Moreover, by standard arguments, given a point $z$ of $\omega(x(\cdot))$ there exists an extremal trajectory starting from $z$ contained in $\omega(x(\cdot))$ and thus the $\omega$-limit set of this trajectory, which is, by Theorem~\ref{p2}, a non trivial periodic  trajectory $\gamma$ of System~\eqref{sys}, also belongs to $\omega(x(\cdot))$. 
By connectedness of $\omega(x(\cdot))$ it turns out that there exist two possibly coinciding periodic trajectories $\gamma^-\leq \gamma^+$ on $S$ (according to the order defined in Remark~\ref{order}) such that $\omega(x(\cdot))$ coincides with the set of periodic trajectories $\gamma$ of \eqref{sys} on $S$ with $\gamma^-\leq\gamma\leq \gamma^+$. These trajectories can be arranged in a finite number of isolated periodic trajectory or isolated continua of periodic trajectories  denoted by $\tilde{\Gamma}_i\subset \omega(x(\cdot))$, $1\leq i\leq q$, and we assume that $\tilde{\Gamma}_1<\cdots<\tilde{\Gamma}_q$ (see Figure~\ref{f-fnotations}). 

By Lemma~\ref{stability0} for $2\leq i\leq q$ the sets $\tilde{\Gamma}_{i-1}$ and $\tilde{\Gamma}_i$ may only be locally attractive or repulsive  on the region of $S$ enclosed by them. Also, according to Theorem~\ref{p2}, they cannot be both repulsive in that region.  In addition, again by Theorem~\ref{p2}, in the connected component of $S\setminus \tilde{\Gamma}_1$ containing $x_*$, the point of $S$ satisfying $c^Tx_*=c^TAx_*=0$ and $c^TA^2x_*>0$, one has that either $\tilde{\Gamma}_1$ is locally attractive or $\omega(x(\cdot))$ is empty. An analogous statement holds for the  connected component of $S\setminus \tilde{\Gamma}_q$ not containing $x_*$. Taking into account the previous remarks it is thus easy to see that there exists $\tilde{\Gamma}_i$ such that, on each connected component of $S\setminus\tilde{\Gamma}_i$, either $\tilde{\Gamma}_i$ is locally attractive or $\omega(x(\cdot))$ is empty.

Notice that the set $P_c\cap \tilde{\Gamma}_i$ is the union of two arcs in $P_c\cap S$ (which reduce to points if $ \tilde{\Gamma}_i$ is a single periodic trajectory). Let $C$ be one of these arcs, then any small enough neighborhood of $C$ in $P_c$ is a transverse section of System~\eqref{sys}. Moreover, let $T_1,T_2$ be positive times such that each periodic trajectory $\gamma$ in $ \tilde{\Gamma}_i$ has period $T_{\gamma}$ satisfying $T_1<T_{\gamma}<T_2$ (the existence of these bounds relies on the continuity of the period of the trajectories of $\tilde{\Gamma}_i$ with respect to the initial datum). The following lemma holds.
\begin{lemma}\label{finite1}
For every small enough $\varepsilon>0$ let $\mathcal{B}_{\varepsilon}$ be the open neighborhood of $C$ in $P_c$ of the points whose distance from $C$ is less than $\varepsilon$. Then there exists $\tau>0$ such that, for every time $t\geq \tau$ with 
$x(t)\in \mathcal{B}_{\varepsilon}$, the smallest $t'>0$ with $x(t+t')\in \mathcal{B}_{\varepsilon}$ satisfies $t'\in (T_1,T_2)$.

\end{lemma}
\begin{proof} 
Let $\bar{\mathcal{B}}_{\varepsilon}$ be the closure of $\mathcal{B}_{\varepsilon}$. By the properties of $ \tilde{\Gamma}_i$ the conclusion is obviously true if, instead of $x(\cdot)$, one considers any extremal trajectory on $S$ starting in  $S\cap\bar{\mathcal{B}}_{\varepsilon}$ and contained in $\omega(x(\cdot))$. The only thing to precise here is that if one starts at the boundary of  $S\cap\bar{\mathcal{B}}_{\varepsilon}$, then, the next intersection of the extremal trajectory with  $S\cap\bar{\mathcal{B}}_{\varepsilon}$ must actually lie in  $S\cap \mathcal{B}_{\varepsilon}$. 

For the general case the argument goes by contradiction. Let us fix $\varepsilon>0$ small and consider an increasing  sequence of times $(t_n)_{n\geq0}$ such that $x(t_n)\in \mathcal{B}_{\varepsilon}$ and  the smallest $t_n'>0$ with $x(t_n+t_n')\in \mathcal{B}_{\varepsilon}$ does not satisfy $t_n'\in (T_1,T_2)$.
Up to a subsequence, $x(t_n+\cdot)$ converges on $[0,T_2]$ to an extremal trajectory on $S$ starting at $\tilde{x}\in S\cap\bar{\mathcal{B}}_{\varepsilon}$ and contained in 
$\omega(x(\cdot))$ which, as explained above, must re-intersect  for the first time $S\cap \mathcal{B}_{\varepsilon}$ at some $t\in (T_1,T_2)$. We have reached a contradiction and that concludes the proof of the lemma.

\end{proof}

To conclude the proof of Proposition~\ref{omega0}
 it is enough to observe that, given $\mathcal{B}_{\varepsilon}$ as in Lemma~\ref{finite1}, the sequence $(t_n)_{n\geq 0}$ of times such that $x(t_n)\in \mathcal{B}_{\varepsilon}$ is infinite and satisfies, for every $n$ large enough, $t_{n+1}-t_n \in [T_1,T_2]$. In particular up to taking $n$ large enough $x(\cdot)$ can be considered as arbitrarily close to an extremal trajectory on the interval $[t_n,t_{n+1}]$ and the latter must be confined on a neighborhood of $\tilde{\Gamma}_i$ that can be made arbitrarily small by letting $\varepsilon$ go to zero. This proves that $\omega(x(\cdot))\subset \tilde{\Gamma}_i$ (and in particular $i=q=1$), concluding the proof of the proposition.

\bibliographystyle{abbrv}
\bibliography{biblio}

\end{document}

%% file: preuve1.pdf_t
\begin{picture}(0,0)%
\includegraphics{preuve1.pdf}%
\end{picture}%
\setlength{\unitlength}{2368sp}%
\begingroup\makeatletter\ifx\SetFigFont\undefined%
\gdef\SetFigFont#1#2#3#4#5{%
  \reset@font\fontsize{#1}{#2pt}%
  \fontfamily{#3}\fontseries{#4}\fontshape{#5}%
  \selectfont}%
\fi\endgroup%
\begin{picture}(3837,2931)(3439,-4243)
\put(3676,-2926){\makebox(0,0)[lb]{\smash{{\SetFigFont{12}{14.4}{\rmdefault}{\mddefault}{\updefault}{\color[rgb]{0,0,0}$\mathcal{C}$}%
}}}}
\put(7261,-2326){\makebox(0,0)[lb]{\smash{{\SetFigFont{12}{14.4}{\rmdefault}{\mddefault}{\updefault}{\color[rgb]{0,0,0}$z=x_1(t_1)=x_2(t_2)$}%
}}}}
\put(4486,-3091){\makebox(0,0)[lb]{\smash{{\SetFigFont{12}{14.4}{\rmdefault}{\mddefault}{\updefault}{\color[rgb]{0,0,0}$\bar x$}%
}}}}
\put(5986,-3436){\makebox(0,0)[lb]{\smash{{\SetFigFont{12}{14.4}{\rmdefault}{\mddefault}{\updefault}{\color[rgb]{0,0,0}$\mathcal{P}_2$}%
}}}}
\put(4816,-4081){\makebox(0,0)[lb]{\smash{{\SetFigFont{12}{14.4}{\rmdefault}{\mddefault}{\updefault}{\color[rgb]{0,0,0}$x_2(\cdot)$}%
}}}}
\put(5296,-1591){\makebox(0,0)[lb]{\smash{{\SetFigFont{12}{14.4}{\rmdefault}{\mddefault}{\updefault}{\color[rgb]{0,0,0}$\mathcal{P}_1$}%
}}}}
\put(3991,-1606){\makebox(0,0)[lb]{\smash{{\SetFigFont{12}{14.4}{\rmdefault}{\mddefault}{\updefault}{\color[rgb]{0,0,0}$x_1(\cdot)$}%
}}}}
\put(6196,-1741){\makebox(0,0)[lb]{\smash{{\SetFigFont{12}{14.4}{\rmdefault}{\mddefault}{\updefault}{\color[rgb]{0,0,0}$\bar{\gamma}(\bar t)$}%
}}}}
\end{picture}%

%% file: preuve0.pdf_t
\begin{picture}(0,0)%
\includegraphics{preuve0.pdf}%
\end{picture}%
\setlength{\unitlength}{2171sp}%
\begingroup\makeatletter\ifx\SetFigFont\undefined%
\gdef\SetFigFont#1#2#3#4#5{%
  \reset@font\fontsize{#1}{#2pt}%
  \fontfamily{#3}\fontseries{#4}\fontshape{#5}%
  \selectfont}%
\fi\endgroup%
\begin{picture}(1917,3999)(4486,-3418)
\put(4501,-676){\makebox(0,0)[lb]{\smash{{\SetFigFont{12}{14.4}{\rmdefault}{\mddefault}{\updefault}{\color[rgb]{0,0,0}$V^-$}%
}}}}
\put(5641,-661){\makebox(0,0)[lb]{\smash{{\SetFigFont{12}{14.4}{\rmdefault}{\mddefault}{\updefault}{\color[rgb]{0,0,0}$A_1$}%
}}}}
\put(5701,-2311){\makebox(0,0)[lb]{\smash{{\SetFigFont{12}{14.4}{\rmdefault}{\mddefault}{\updefault}{\color[rgb]{0,0,0}$A_p$}%
}}}}
\put(4951,314){\makebox(0,0)[lb]{\smash{{\SetFigFont{12}{14.4}{\rmdefault}{\mddefault}{\updefault}{\color[rgb]{0,0,0}$U^-$}%
}}}}
\put(4771,-3196){\makebox(0,0)[lb]{\smash{{\SetFigFont{12}{14.4}{\rmdefault}{\mddefault}{\updefault}{\color[rgb]{0,0,0}$U^+$}%
}}}}
\put(4516,-1246){\makebox(0,0)[lb]{\smash{{\SetFigFont{12}{14.4}{\rmdefault}{\mddefault}{\updefault}{\color[rgb]{0,0,0}$x_*$}%
}}}}
\end{picture}%

%% file: lemma.pdf_t
\begin{picture}(0,0)%
\includegraphics{lemma.pdf}%
\end{picture}%
\setlength{\unitlength}{1973sp}%
\begingroup\makeatletter\ifx\SetFigFont\undefined%
\gdef\SetFigFont#1#2#3#4#5{%
  \reset@font\fontsize{#1}{#2pt}%
  \fontfamily{#3}\fontseries{#4}\fontshape{#5}%
  \selectfont}%
\fi\endgroup%
\begin{picture}(7472,4809)(2643,-5458)
\put(8447,-3631){\makebox(0,0)[lb]{\smash{{\SetFigFont{12}{14.4}{\rmdefault}{\mddefault}{\updefault}{\color[rgb]{0,0,0}$S_*^+$}%
}}}}
\put(3976,-2401){\makebox(0,0)[lb]{\smash{{\SetFigFont{12}{14.4}{\rmdefault}{\mddefault}{\updefault}{\color[rgb]{0,0,0}$x(t_0)=x_*$}%
}}}}
\put(4006,-5146){\makebox(0,0)[lb]{\smash{{\SetFigFont{12}{14.4}{\rmdefault}{\mddefault}{\updefault}{\color[rgb]{0,0,0}$x(0)$}%
}}}}
\put(3991,-976){\makebox(0,0)[lb]{\smash{{\SetFigFont{12}{14.4}{\rmdefault}{\mddefault}{\updefault}{\color[rgb]{0,0,0}$P_c\cap S$}%
}}}}
\put(9497,-5221){\makebox(0,0)[lb]{\smash{{\SetFigFont{12}{14.4}{\rmdefault}{\mddefault}{\updefault}{\color[rgb]{0,0,0}$x(0)$}%
}}}}
\put(9482,-1876){\makebox(0,0)[lb]{\smash{{\SetFigFont{12}{14.4}{\rmdefault}{\mddefault}{\updefault}{\color[rgb]{0,0,0}$x(t_0)=x_*$}%
}}}}
\put(9437,-1036){\makebox(0,0)[lb]{\smash{{\SetFigFont{12}{14.4}{\rmdefault}{\mddefault}{\updefault}{\color[rgb]{0,0,0}$P_c\cap S$}%
}}}}
\put(2941,-3706){\makebox(0,0)[lb]{\smash{{\SetFigFont{12}{14.4}{\rmdefault}{\mddefault}{\updefault}{\color[rgb]{0,0,0}$S_*^+$}%
}}}}
\end{picture}%

%% file: periodic-order.pdf_t
\begin{picture}(0,0)%
\includegraphics{periodic-order.pdf}%
\end{picture}%
\setlength{\unitlength}{1973sp}%
\begingroup\makeatletter\ifx\SetFigFont\undefined%
\gdef\SetFigFont#1#2#3#4#5{%
  \reset@font\fontsize{#1}{#2pt}%
  \fontfamily{#3}\fontseries{#4}\fontshape{#5}%
  \selectfont}%
\fi\endgroup%
\begin{picture}(6315,5345)(2581,2525)
\put(6361,7424){\makebox(0,0)[lb]{\smash{{\SetFigFont{12}{14.4}{\rmdefault}{\mddefault}{\updefault}{\color[rgb]{0,0,0}$x_*$}%
}}}}
\put(8341,7034){\makebox(0,0)[lb]{\smash{{\SetFigFont{12}{14.4}{\rmdefault}{\mddefault}{\updefault}{\color[rgb]{0,0,0}$\tilde{\Gamma}_1$}%
}}}}
\put(8716,6299){\makebox(0,0)[lb]{\smash{{\SetFigFont{12}{14.4}{\rmdefault}{\mddefault}{\updefault}{\color[rgb]{0,0,0}$\tilde{\Gamma}_2$}%
}}}}
\put(8881,5054){\makebox(0,0)[lb]{\smash{{\SetFigFont{12}{14.4}{\rmdefault}{\mddefault}{\updefault}{\color[rgb]{0,0,0}$\tilde{\Gamma}_3$}%
}}}}
\put(3781,7304){\makebox(0,0)[lb]{\smash{{\SetFigFont{12}{14.4}{\rmdefault}{\mddefault}{\updefault}{\color[rgb]{0,0,0}$\gamma^-$}%
}}}}
\put(2596,5639){\makebox(0,0)[lb]{\smash{{\SetFigFont{12}{14.4}{\rmdefault}{\mddefault}{\updefault}{\color[rgb]{0,0,0}$\gamma^+$}%
}}}}
\put(2626,3989){\makebox(0,0)[lb]{\smash{{\SetFigFont{12}{14.4}{\rmdefault}{\mddefault}{\updefault}{\color[rgb]{0,0,0}$S$}%
}}}}
\end{picture}%